\documentclass{amsart}
\usepackage{amssymb}
\usepackage{txfonts}
\usepackage{amsfonts}
\usepackage{color}
\usepackage{amsmath}

\def\eqref#1{({\color{red}\ref{#1}})}

\def\eqref#1{(\ref{#1})}

\def\eqref#1{(\ref{#1})}

\newtheorem{theorem}{Theorem}[section]
\newtheorem{lemma}[theorem]{Lemma}
\newtheorem{proposition}[theorem]{Proposition}

\newtheorem{remark}[theorem]{Remark}%[section]
\numberwithin{equation}{section}

\newcommand\RR{{{\mathbb R}}}

\let\wt=\widetilde

\begin{document}
\title[The non cutoff Kac's equation]
{Gevrey regularizing effect of the Cauchy problem\\
for non-cutoff homogeneous Kac's equation}
\author[N. Lekrine]{Nadia Lekrine}
\address{Universit\'e de Rouen, UMR 6085-CNRS, Math\'ematiques, Avenue de
l'Universit\'e,\,\, BP.12, 76801 Saint Etienne du Rouvray, France}
\email{lekrinenadia@yahoo.fr}
\author[C.-J. Xu]{Chao-Jiang XU}
\address{Universit\'e de Rouen, UMR 6085-CNRS, Math\'ematiques, Avenue de
l'Universit\'e,\,\, BP.12, 76801 Saint Etienne du Rouvray, France}
\address{and School of mathematics, Wuhan University, 430072, Wuhan, China }
\email{Chao-Jiang.Xu\@@univ-rouen.fr}
\date{Revision - 25/08/2009}

\begin{abstract}
In this work, we consider a spatially homogeneous Kac's equation
with a non cutoff cross section. We prove that the weak solution of
the Cauchy problem is in {the Gevrey }class for positive time. This
is a Gevrey regularizing effect for non smooth initial datum. The
proof {  relies on} the Fourier analysis of Kac's operators and { on
an} exponential type {  mollifier}.
\end{abstract}

\maketitle

\section{Introduction}

\label{section1}

In this work, we consider the following Cauchy problem for spatially
homogeneous non linear Kac's equation,
\begin{equation}
\left\{
\begin{array}{ll}
\frac{\partial f}{\partial t}=K(f,\,\,f), & v\in {{\mathbb{R}}},\,\,t>0, \\
f|_{t=0}=f_{0}\,. &
\end{array}
\right.  \label{1.1}
\end{equation}%
where $f=f(t,\,v)$ is the nonnegative density distribution function of
particles with velocity $v\in {{\mathbb{R}}}$ at time $t$. The right hand
side of equation (\ref{1.1}) is given by Kac's bilinear collisional operator
\begin{equation*}
K(f,\,\,g)=\int_{{{\mathbb{R}}}}\int_{-\pi /2}^{\pi /2}\beta (\theta
)\left\{ f(v_{\ast }^{\prime })g(v^{\prime })-f(v_{\ast })g(v)\right\}
d\theta dv_{\ast }\,,
\end{equation*}%
where
\begin{equation*}
v^{\prime }=v\,\cos \theta -v_{\ast }\,\sin \theta ,\quad \,\,v_{\ast
}^{\prime }=v\,\sin \theta +v_{\ast }\,\cos \theta .
\end{equation*}%
We suppose that the cross-section kernel is non cut-off. To simplify the
notations, we suppose (see \cite{D95,Des-Fu-Ter} for the precise description
of cross-section kernel) that
\begin{equation}
\beta (\theta )=C_{0}\frac{|\cos \theta |\,\,\,\,\,}{\,\,|\sin \theta
|^{1+2s}},\,\,\,\,\,-\frac{\pi }{2}\leq \theta \leq \frac{\pi }{2}\,,
\label{1.1+1}
\end{equation}%
where $0<s<1$ and $C_{0}>0$, then
\begin{equation*}
\int_{-\pi /2}^{\pi /2}\beta (\theta )d\theta =+\infty ,
\end{equation*}%
and
\begin{equation}
\left\{\begin{split}
& \int_{-\pi /2}^{\pi /2}\beta (\theta )\,|\theta |\,d\theta =C_{s}<+\infty
,\quad 0<s<1/2, \\
& \int_{-\pi /2}^{\pi /2}\beta (\theta )\,\theta ^{2}\,d\theta
=C_{s}<+\infty ,\quad 0<s<1.
\end{split}%
\right.  \label{1.1+2}
\end{equation}

Hereafter, use the following function spaces: For $1\leq p\leq
+\infty ,\ell \in {{\mathbb{R}}}$,
\begin{equation*}
L_{\ell }^{p}({{\mathbb{R}}})=\Big\{f;\,\,\Vert f\Vert _{L_{\ell }^{p}}=\Big(%
\int_{{{\mathbb{R}}}}|\langle v\rangle ^{\ell }f(v)|^{p}dv\Big)%
^{1/p}<+\infty \Big\}
\end{equation*}%
where $\langle v\rangle =(1+|v|^{2})^{1/2}$.
\begin{equation*}
L\log L({{\mathbb{R}}})=\Big\{f;\,\,\Vert f\Vert _{L\log L}=\int_{{{\mathbb{R%
}}}}|f(v)|\log (1+|f(v)|)dv<+\infty \Big\}.
\end{equation*}%
For $k,\ell \in {{\mathbb{R}}}$,
\begin{equation*}
H_{\ell }^{k}({{\mathbb{R}}})=\left\{ f\in {\mathcal{S}}^{\prime }({{\mathbb{%
R}}});\langle v\rangle ^{\ell }f\in H^{k}({{\mathbb{R}}})\right\} .
\end{equation*}

We assume that the initial datum $f_0\equiv \hspace{-3mm} / \: 0$ satisfies
the natural boundedness on the mass, energy and entropy, that is,
\begin{equation}  \label{1.2+0}
f_0\geq 0,\,\,\,\, \int_{{{\mathbb{R}}}}f_0(v)\Big(1+|v|^2+\log(1+f_0(v))%
\Big)dv<+\infty.
\end{equation}

In \cite{D95}, L. Desvillettes has proved the existence of a nonnegative
weak solution to the Cauchy problem (\ref{1.1}), (see also \cite%
{graham-meleard} by using a stochastic calculus),
\begin{equation}  \label{1.2}
f\in L^\infty ([0, +\infty[; L^1_k({{\mathbb{R}}}))\, ,
\end{equation}
if $f_0\in L^1_k({{\mathbb{R}}})$ for some $k\geq 2$. The weak solution
satisfies the conservation of mass
\begin{equation}  \label{1.3}
\int_\RR f(t, v) dv = \int_\RR f_0(v) dv,\quad \forall t>0,
\end{equation}
the conservation of energy
\begin{equation}  \label{1.3b}
\int_\RR f(t, v) |v|^2dv = \int_\RR f_0(v)|v|^2 dv ,\quad \forall t>0,
\end{equation}
and also the entropy inequality
\begin{equation}  \label{1.3c}
\int_\RR f(t, v) \log f(t, v) dv \leq \int_\RR f_0(v) \log f_0(v) dv ,\quad
\forall t>0,
\end{equation}
but does not conserve the momentum.

L. Desvillettes proved also in \cite{D95} (see also \cite{fournier}), the $%
C^{\infty }$-regularity of weak solutions if $f_{0}\in L_{\ell }^{1}({{%
\mathbb{R}}})$ for any $\ell \in {\mathbb{N}}$. This regularizing
effect properties is now well-known for non cut-off homogeneous
Boltzmann equations (see also \cite{al-2,
al-3,desv-wen1,MUXY-DCDS}).

In this work, we consider the higher order regularity, the Gevrey
regularity of solutions of the Cauchy problem (\ref{1.1}). We start
by recalling  the definition of the Gevrey class functions. $u\in
G^{\alpha }({{\mathbb{R}}}^{n})$ (the Gevrey class function space
with index $\alpha )$, if for $\alpha \geq 1,$ there exists $C>0$
such that for any $k\in {\mathbb{N}}$,
\begin{equation*}
\Vert D^{k}u\Vert _{L^{2}({{\mathbb{R}}}^{n})}\leq C^{k+1}(k!)^{\alpha },
\end{equation*}%
or equivalently, there exists $c_{0}>0$ such that $e^{c_{0}\langle D\rangle
^{1/\alpha }}u\in L^{2}({{\mathbb{R}}}^{n})$, where
\begin{equation*}
\langle D\rangle =(1+|D_{v}|^{2})^{1/2},\,\quad \,\Vert D^{k}u\Vert _{L^{2}({%
{\mathbb{R}}}^{n})}^{2}=\sum_{|\beta |=k}\Vert D^{\beta }u\Vert _{L^{2}({{%
\mathbb{R}}}^{n})}^{2}.
\end{equation*}%
Note that $G^{1}({{\mathbb{R}}}^{n})$ is the usual analytic function space.
If $0<\alpha <1$, the above definition gives the ultra-analytical function
class. Recall that we give here the Gevrey class functions on ${{\mathbb{R}}}%
^{n} $, and so we can use the Fourier transformation and give an
equivalent definition by using {  a Fourier} multiplier
$e^{c_{0}\langle D\rangle ^{1/\alpha }}$, we can also replace $
L^{2}$-norm by $L^{\infty}$-norm.

Our result on the Gevrey regularity can be stated as follows.

\begin{theorem}
\label{theo1} Assume that the initial datum $f_{0}\in
L_{2+2s}^{1}\cap L\log L({{\mathbb{R}}})$, and the cross-section
$\beta $ satisfy (\ref{1.1+1}) with $0<s<\frac{1}{2}$. For $T_0>0$,
if $f\in L^{\infty }([0, T_0] ;L_{2+2s}^{1}\cap L\log
L({{\mathbb{R}}}))$ is a nonnegative weak solution of the Cauchy
problem (\ref{1.1}), then for any $0<s^{\prime }<s$, there exists
$0<T_*\leq T_0$ such that
\begin{equation*}
\displaystyle f(t,\cdot )\in G^{^{\frac{1}{2s^{\prime }}}}({{\mathbb{R}}})
\end{equation*}
for any $0<t\leq T_*$.
\end{theorem}

\begin{remark}
The above results is a smoothing effect property in the Gevrey class
for the Cauchy problem. We suppose nothing about regularity and high
order moment controls for the initial datum.
\end{remark}

Recall that  Kac's equation  is obtained when one considers
radially symmetric solutions of the spatially homogeneous Boltzmann
equation for Maxwellian molecules (see \cite{D95}). The Cauchy
problem for the spatially homogeneous Boltzmann equation is defined
by :
\begin{equation}
\frac{\partial g}{\partial t}=Q(g,\,g),\enskip v\in {{\mathbb{R}}}%
^{3},\,\,\,t>0\,;\hspace{1cm}g|_{t=0}=g_{0}\,,  \label{1.4}
\end{equation}%
where the Boltzmann collision operator $Q(g,\,f)$ is a bi-linear functional
given by
\begin{equation}
Q(g,\,f)=\int_{{{\mathbb{R}}}^{3}}\int_{\mathbb{S}^{2}}B\left( {v-v_{\ast }}%
,\sigma \right) \left\{ g(v_{\ast }^{\prime })f(v^{\prime })-g(v_{\ast
})f(v)\right\} d\sigma dv_{\ast }\,,  \label{1.5}
\end{equation}%
for $\sigma \in \mathbb{S}^{2}$ and where
\begin{equation*}
v^{\prime }=\frac{v+v_{\ast }}{2}+\frac{|v-v_{\ast }|}{2}\sigma
,\,\,\,v_{\ast }^{\prime }=\frac{v+v_{\ast }}{2}-\frac{|v-v_{\ast }|}{2}%
\sigma \,.
\end{equation*}%
The non-negative function $B(z,\sigma )$ called the Boltzmann collision
kernel depends only on $|z|$ and the scalar product $<\frac{z}{|z|},\sigma >$%
. In most of the cases, the collision kernel $B$ can not be expressed
explicitly. However, to capture its main property, it can be assumed to be
in the form
\begin{equation*}
B(|v-v_{\ast }|,\cos \theta )=\Phi (|v-v_{\ast }|)b(\cos \theta ),\,\,\,\cos
\theta =\big<\frac{v-v_{\ast }}{|v-v_{\ast }|}\,,\,\sigma \big>,\,\,\,-\frac{%
\pi }{2}\leq \theta \leq \frac{\pi }{2}.
\end{equation*}%
The Maxwellian case corresponds to $\Phi \equiv 1$. Except for hard sphere
model, the function $b(\cos \theta )$ has a singularity at $\theta =0$. We
assume that
\begin{equation}
\sin \theta \,b(\cos \theta )\,\,\approx \,\,K\theta ^{-1-2s}\,\,\,%
\mbox{when}\,\,\theta \rightarrow 0,  \label{1.6}
\end{equation}%
where $K>0,0<s<1$. Remark that the solution of Boltzmann equation satisfies 
also the conservation of mass, energy and the entropy inequality. 

A function $g$ is radially symmetric with respect to $v\in
{{\mathbb{R}}}^{3} $, if it satisfy the property
$$
g(t, v)=g(t, Av),\quad v\in \mathbb{R}^3
$$
for any rotation $A$ in $\mathbb{R}^3$. We proved the following
results.
\begin{theorem}
\label{theo2} Assume that the initial datum $g_{0}\in
L_{2+2s}^{1}\cap L\log L({{\mathbb{R}}}^{3}), g_0\geq 0$ is radially
symmetric. Let $\Phi \equiv 1$ and  let $b$ satisfy (\ref{1.6})
with $0<s<\frac{1}{2}$. If $g$ is a nonnegative radially symmetric
weak solution of the Cauchy problem (\ref{1.4}) such that
$g\in L^{\infty }(]0,+\infty \lbrack ;L_{2+2s}^{1}\cap L\log L({{\mathbb{R}}}%
^{3}))$ , then
\begin{equation*}
g(t,\,\cdot \,)\in G^{^{\frac{1}{2s^{\prime }}}}({{\mathbb{R}}}_{v}^{3})
\end{equation*}%
for any $t>0$ and any $0<s^{\prime }<s$.
\end{theorem}

Remark that for the non cut-off spatially homogeneous Boltzmann equation, we
have the $H^{\infty }$-regularizing effect of weak solutions (see also \cite%
{desv-wen1,HMUY,MUXY-DCDS, al-3}). Namely if $f$  is a weak
solution of the Cauchy problem (\ref{1.4}) and the cross section $b$ satisfy
(\ref{1.6}), then we have $f(t,\cdot )\in H^{+\infty}
({\mathbb{R}})$ for any $0<t$.

Notice that, for the Boltzmann equation, the local solutions having
the Gevrey regularity have been constructed in \cite{ukai} for
initial data having higher Gevrey regularity, and the propagation of
Gevrey regularity for solutions of Boltzmann equation is studied in
\cite{Des-Fu-Ter}. The result given here is concerned with the
production of the Gevrey regularity for weak solutions whose initial
data have no assumption on the regularity. This regularizing effect
property of { the Cauchy }problem is analogous to the results of
\cite{MUXY-DCDS} where linearized Boltzmann equation is considered.
In \cite{mo-xu2}, we have the ultra-analytical regularizing effect
of the Cauchy problem in $G^{{\frac{1}{2}}}({{\mathbb{R}}}^{3})$ for
the homogeneous Landau equations, which is optimal {  as seen from
the Cauchy} problem of heat equation.

%%%%%%%%%%%%%%%%%%%%%%%%%%%%%%%%%%%%%%%%%%%%%%%%%%%%%%%%%%%%%%%%%%%%%%%
%%%%%%%%%%%%%%%%%%%%%%%%%%%%%%%%%%%%%%%%%%%%%%%%%%%%%%%%%%%%%%%%%%%%%%%
\vskip0.5cm

\section{Fourier analysis of Kac's operators}

\label{section2} %\setcounter{equation}{0}
\smallskip

We will now be interested in studying the Fourier analysis of the
Kac's collision operator. This is a key step in the regularity
analysis of weak solutions. For simplification of notations, we use
 $(\cdot \,,\,\cdot )$ instead of $(\cdot \,,\,\cdot
)_{L^{2}({{\mathbb{R}}}_{v})}$. We have firstly the following
coercivity estimate deduced from the non cut-off of collision
kernel.
\begin{proposition}\label{prop2.0} 
Assume that the cross-section is non cut-off, satisfies the
assumption (\ref{1.1+1}). Let $f\geq 0, f\neq0, f\in L^{1}_{1}({{\mathbb{R}}}%
) \cap L\log L({{\mathbb{R}}})$, then there exists a constant
$c_{f}>0$, depending only on $\beta,\|f\|_{L^1_1}$, and $\|
f\|_{LLogL}$, such that
\begin{equation}  \label{2.1}
-\Big(K(f,\, g),\,\, g\Big)\geq c_f \|g\|^2_{H^s({{\mathbb{R}}}_v)}-
C\|f\|_{L^1}\|g\|^2_{L^2}
\end{equation}
for any smooth function $g\in H^1({{\mathbb{R}}})$.
\end{proposition}

\begin{remark}\label{rema2.1}
In the proof of Proposition \ref{prop2.0} , the following properties are essential 
(see (44) in \cite{al-1})

(H-1) there exists a $r > 0$ such that $\int_{\{v\in\RR; |v|\leq r\}} f(v)dv \geq \frac{3}{4} \|f \|_{L^{1}}$

(H-2) there exists a $\delta > 0$ such that $\int_{A}f(v)dv <\frac{1}{4}\|f\|_{L^{1}}$ for any measurable 
set $A\subset\RR$ satsifying $|A| < \delta$.

As stated in Lemma 2.2 of \cite{Des-Fu-Ter}(p.1738), Lebesgue's 
dominated convergence theorem shows that both properties follow only from the assumption
$f \in L^{1}$. However, the proof of Theorem \ref{theo1} and Theorem \ref{theo2} require that $r$ 
and $\delta$ can be chosen uniformly with respect to $t$ if Proposition \ref{prop2.0} 
is applied to solution $f(t, v)$. Under the conservation of mass (\ref{1.3}), (H-1) and (H-2), respectively, 
follow from (\ref{1.3b}) and (\ref{1.3c}), respectively. In the proof of Theorem \ref{1.3}, 
the property (H-2) will be checked directly without the entropy inequality (see Lemma \ref{lemm2} below).
\end{remark}

Recall the following weak formulation for collision operators
\begin{equation*}
\Big(K(f,\,g),\,\,h\Big)=\iint_{{{\mathbb{R}}}^{2}}\int_{-\frac{\pi }{2}}^{%
\frac{\pi }{2}}\beta (\theta )f(v_{\ast })g(v)\Big(h(v^{\prime })-h(v)\Big)%
d\theta dv_{\ast }dv,
\end{equation*}%
for suitable functions $f,g,h$ with reals values. Then
\begin{eqnarray*}
\Big(-K(f,\,g),\,\,g\Big) &=&\frac{1}{2}\iint_{{{\mathbb{R}}}^{2}}\int_{-%
\frac{\pi }{2}}^{\frac{\pi }{2}}\beta (\theta )f(v_{\ast })\Big(g(v^{\prime
})-g(v)\Big)^{2}d\theta dv_{\ast }dv \\
&&-\frac{1}{2}\iint_{{{\mathbb{R}}}^{2}}\int_{-\frac{\pi }{2}}^{\frac{\pi }{2%
}}\beta (\theta )f(v_{\ast })\Big(g(v^{\prime })^{2}-g(v)^{2}\Big)d\theta
dv_{\ast }dv\,.
\end{eqnarray*}%
The second term of right hand side can be estimated by using the
Cancellation lemma of \cite{al-1}. But in  the Maxwellien case, by
an appropriate change of variable, we then have,
\begin{eqnarray*}
&&\left\vert \frac{1}{2}\int_{{{\mathbb{R}}}^{2}}\iint_{-\frac{\pi }{2}}^{%
\frac{\pi }{2}}\beta (\theta )f(v_{\ast })\Big(g(v^{\prime })^{2}-g(v)^{2}%
\Big)d\theta dv_{\ast }dv\right\vert \\
&=&\left\vert \frac{1}{2}\iint_{{{\mathbb{R}}}^{2}}\int_{-\frac{\pi }{2}}^{%
\frac{\pi }{2}}\beta (\theta )f(v_{\ast })g(v)^{2}\Big(\frac{1}{\cos \theta }%
-1\Big)d\theta dv_{\ast }dv\right\vert \\
&\leq &C\iint_{{{\mathbb{R}}}^{2}}\int_{-\frac{\pi }{2}}^{\frac{\pi }{2}}%
\big|\sin (\theta )\big|^{-1-2s}\Big|\sin \Big(\frac{\theta }{2}\Big)\Big|%
^{2}|f(v_{\ast })|g(v)^{2}d\theta dv_{\ast }dv \\
&\leq &C\Vert f\Vert _{L^{1}}\Vert g\Vert _{L^{2}}^{2}\,.
\end{eqnarray*}%
The coercivity term in $H^{s}$ is deduced from the following positive term,
\begin{equation*}
\frac{1}{2}\int_{{{\mathbb{R}}}^{2}}\int_{-\frac{\pi }{2}}^{\frac{\pi }{2}%
}\beta (\theta )f(v_{\ast })\Big(g(v^{\prime })-g(v)\Big)^{2}d\theta
dv_{\ast }dv.
\end{equation*}%
Here we need the Bobylev formula, i. e. the Fourier transform of
collision operators :
\begin{equation}
{\mathcal{F}}\Big(K(f,\,g)\Big)(\xi )=\frac{1}{2\pi }\int_{-\frac{\pi }{2}}^{%
\frac{\pi }{2}}\beta (\theta )\left\{ \hat{f}(\xi \,\sin {\theta })\hat{g}%
(\xi \,\cos {\theta })-\hat{f}(0)\hat{g}(\xi )\right\} d\theta \,,
\label{2.2}
\end{equation}%
for suitable functions $f$ and $g$ and by using  both properties (1) and (2) and the 
unifom integrability of $f_{t}$. (see \cite{al-1,al-3,MUXY-DCDS}).
From the above formula, we can get also the following upper bound
estimates (see \cite{HMUY,MUXY-DCDS}). For $m,\ell \in
{{\mathbb{R}}}$, and for suitable functions $f,\,g,$ we have
\begin{equation}
\Vert K(f,\,g)\Vert _{H_{\ell }^{m}({{\mathbb{R}}}_{v})}\leq C\Vert f\Vert
_{L_{\ell ^{+}+2s}^{1}({{\mathbb{R}}}_{v})}\Vert g\Vert _{H_{(\ell
+2s)^{+}}^{m+2s}({{\mathbb{R}}}_{v})}\,,  \label{estimate-E}
\end{equation}%
where $\alpha ^{+}=\max \{\alpha ,0\}$.

\bigbreak To study the Gevrey regularity of the weak solution, as in \cite%
{MUXY-DCDS,mo-xu2}, we consider the exponential type mollifier. For $%
0<\delta <1$, $c_{0}>0$ and $0<s^{\prime }<s$, we set
\begin{equation*}
G_{\delta }(t,\,\xi )=\frac{e^{c_{0}\,t\,\langle \xi \rangle ^{2s^{\prime }}}%
}{1+\delta e^{c_{0}\,t\,\langle \xi \rangle ^{2s^{\prime }}}}
\end{equation*}%
where
\begin{equation*}
\langle \xi \rangle =(1+|\xi |^{2})^{\frac{1}{2}},\,\,\,\xi \in {{\mathbb{R}}%
}.
\end{equation*}%
Then, for any $0<\delta <1$,
\begin{equation}
G_{\delta }(t,\,\xi )\in L^{\infty }(]0,T[\times {{\mathbb{R}}}),
\label{2.4}
\end{equation}%
and
\begin{equation}
\lim_{\delta \rightarrow 0}G_{\delta }(t,\,\xi )=e^{c_{0}\,t\,\langle \xi
\rangle ^{2s^{\prime }}}.  \label{2.5}
\end{equation}
Denote by $G_{\delta }(t,\,D_{v}),$ the Fourier multiplier of symbol $%
G_{\delta }(t,\,\xi )$,
\begin{equation*}
G_{\delta }\,g(t,\,v)=G_{\delta
}(t,\,D_{v})g(t,\,v)={\mathcal{F}}_{\xi\, \rightarrow\, v}^{-1}\big(
G_{\delta }(t,\,\xi )\hat{g}(t,\xi )\big) .
\end{equation*}
Then our aim is to prove the uniform boundedness (with respect to
$0<\delta
<1$) of the term $\Vert G_{\delta }(t,\,D_{v})f(t,\,\cdot )\Vert _{L^{2}({{%
\mathbb{R}}})}$ for the weak solution of the Cauchy problem
(\ref{1.1}). In what follows, we will use the same notation
$G_{\delta }$ for the
pseudo-differential operators $G_{\delta }(t,D_{v})$ and also its symbol $%
G_{\delta }(t,\xi )$.

\begin{lemma}
\label{lemm2.1} Let $T>0, c_0>0$, We have that for any $0<\delta<1$ and $%
0\leq t\leq T,\, \xi\in{{\mathbb{R}}}$,
\begin{equation*}
\left|\partial_t G_\delta(t, \xi)\right|\leq c_0\langle \xi\rangle
^{2s^{\prime}} G_\delta(t, \xi) ,
\end{equation*}
\begin{equation*}
\left|\partial_{\xi} G_\delta(t, \xi)\right|\leq {2s^{\prime}}\, c_0\, t\,
\langle \xi\rangle^{{2s^{\prime}}-1} G_\delta(t, \xi)\,
\end{equation*}
and
\begin{equation*}
\left|\partial^2_{\xi} G_\delta(t, \xi)\right|\leq C \langle
\xi\rangle^{2(2s^{\prime}-1)} G_\delta(t, \xi)\,
\end{equation*}
with $C>0$ independent of $\delta$.
\end{lemma}

In fact, we have the following formulas
\begin{equation}  \label{2.6}
\partial_t G_\delta(t, \xi)=c_0\langle \xi\rangle ^{2s^{\prime}} G_\delta(t,
\xi) \frac {1} {1+\delta e^{c_0 t\langle \xi\rangle ^{2s^{\prime}}}},
\end{equation}
\begin{equation}  \label{2.7}
\partial_{\xi} G_\delta(t, \xi)=2s^{\prime}\, c_0\, t\,
(1+|\xi|^2)^{s^{\prime}-1}\xi\, G_\delta(t, \xi)\,\frac {1} {1+\delta e^{c_0
t\langle \xi\rangle ^{2s^{\prime}}}},
\end{equation}
and
\begin{equation}  \label{2.8}
\begin{split}
&\partial^2_{\xi} G_\delta(t, \xi)= \Big(2s^{\prime}\, c_0\, t\,
(1+|\xi|^2)^{s^{\prime}-1}\xi\Big)^2\, G_\delta(t, \xi)\,\frac {1-\delta
e^{c_0 t\langle \xi\rangle ^{2s^{\prime}}}} {\left(1+\delta e^{c_0 t\langle
\xi\rangle ^{2s^{\prime}}}\right)^2} \\
&\,\,\,\,+\, 2s^{\prime}\,c_0\, t \Big((1+|\xi|^2)^{s^{\prime}-1}+
2(s^{\prime}-1)\xi^2(1+|\xi|^2)^{s^{\prime}-2}\Big)\, G_\delta(t, \xi)\,
\frac {1} {1+\delta e^{c_0 t\langle \xi\rangle ^{2s^{\prime}}}}\,.
\end{split}%
\end{equation}

\begin{lemma}
\label{lemm2.2} There exists $C>0$ such that for all $0<\delta<1$ and $%
\xi\in {{\mathbb{R}}}$, we have,
\begin{equation}  \label{2.9}
|G_{\delta}(\xi)-G_{\delta}(\xi \cos\theta)|\leq C
\sin^{2}(\theta/2)\langle\xi \rangle^{2s^{\prime}}
G_{\delta}(\xi\cos\theta)\, G_{\delta}(\xi\sin\theta),
\end{equation}
and
\begin{equation}  \label{2.10}
\left|\big(\partial_\xi G_{\delta}\big)(\xi)-\big(\partial_\xi G_{\delta}%
\big)(\xi \cos\theta)\right|\leq C \sin^{2}(\theta/2)\langle\xi
\rangle^{(4s^{\prime}-1)^+}
G_{\delta}(\xi\cos\theta)G_{\delta}(\xi\sin\theta),
\end{equation}
where $(4s^{\prime}-1)^+=\max\{4s^{\prime}-1, 0\}$.
\end{lemma}

\begin{proof}
For the estimate (\ref{2.9}), we have, by using the Taylor
formula
\begin{equation*}
G_{\delta }(\xi )-G_{\delta }(\xi \cos \theta )=\big(\xi -\xi \cos \theta %
\big)\int_{0}^{1}\big(\partial _{\xi }G_{\delta }\big)(\xi \cos \theta +\tau
(\xi -\xi \cos \theta ))d\tau \,
\end{equation*}%
where $\xi _{\tau }=\xi \cos \theta +\tau (\xi -\xi \cos \theta )$. Then (%
\ref{2.7}) implies
\begin{equation*}
|G_{t,\,\delta }(\xi )-G_{\delta }(t,\,\xi \cos \theta )|\leq 4s^{\prime
}\,c_{0}\,t\,|\xi |\,\sin ^{2}(\theta /2)\int_{0}^{1}G_{\delta }(t,\,\xi
_{\tau })\langle \xi _{\tau }\rangle ^{{2s^{\prime }}-1}d\tau \,.
\end{equation*}%
For $0\leq \tau \leq 1$ and $-\pi /4\leq \theta \leq \pi /4$,
\begin{equation*}
\frac{\sqrt{2}}{2}|\xi |\leq |\xi _{\tau }|=|\xi \cos \theta +\tau (\xi -\xi
\cos \theta )|\leq |\xi |,
\end{equation*}%
which implies, for $0<2s^{\prime }<1$, that there exists $C_{s^{\prime }}>0$
such that
\begin{equation*}
\langle \xi _{\tau }\rangle ^{2s^{\prime }}\leq \langle \xi \rangle
^{2s^{\prime }},\,\,\,\,\,\,\,\,\langle \xi _{\tau }\rangle ^{2s^{\prime
}-1}\leq C_{s}\langle \xi \rangle ^{2s^{\prime }-1}.
\end{equation*}%
On the other hand, $G_{\delta }(t,\xi )=G_{\delta }(t,|\xi |)$ is
increasing with respect to $|\xi |$, since for $\xi >0$, $\partial
_{\xi }G_{\delta }(t,\xi )>0$, then
\begin{equation*}
G_{\delta }(t,\,\xi _{\tau })\leq G_{\delta }(t,\,\xi ).
\end{equation*}%
By using
\begin{equation*}
|\xi |^{2}=|\xi \cos \theta |^{2}+|\xi \sin \theta |^{2},
\end{equation*}%
and
\begin{equation*}
(1+a+b)^{2s^{\prime }}\leq (1+a)^{2s^{\prime }}+(1+b)^{2s^{\prime
}},\,\,\,\,\,\,(1+\delta e^{\alpha })(1+\delta e^{\beta })\leq 3(1+\delta
e^{\alpha +\beta }),
\end{equation*}%
we get
\begin{equation}
G_{\delta }(\xi )\leq 3G_{\delta }(\xi \,\cos \theta )G_{\delta }(\xi \,\sin
\theta ).  \label{2.11}
\end{equation}%
Thus
\begin{equation*}
|G_{\delta }(\xi )-G_{\delta }(\xi \,\cos \theta )|\leq C\sin ^{2}(\theta
/2)\langle \xi \rangle ^{2s^{\prime }}G_{\delta }(\xi \,\cos \theta
)G_{\delta }(\xi \sin \theta ).
\end{equation*}%
We have proved the estimate (\ref{2.9}) when $|\theta |\leq \pi /4$. If $\pi
/4\leq |\theta |\leq \pi /2$, we have
\begin{eqnarray*}
&&|G_{\delta }(\xi )-G_{\delta }(\xi \,\cos \theta )|\leq |G_{\delta }(\xi
)|+|G_{\delta }(\xi \,\cos \theta )|\leq 2|G_{\delta }(\xi )| \\
&\leq &6\,\,G_{\delta }(\xi \,\cos \theta )G_{\delta }(\xi \sin \theta )\leq
C\sin ^{2}(\theta /2)\,G_{\delta }(\xi \,\cos \theta )G_{\delta }(\xi \sin
\theta ).
\end{eqnarray*}%
For the estimate (\ref{2.10}), by using (\ref{2.8}), we have that if $%
|\theta |\leq \pi /4$,
\begin{eqnarray*}
&&\left\vert \big(\partial _{\xi }G_{\delta }\big)(\xi )-\big(\partial _{\xi
}G_{\delta }\big)(\xi \cos \theta )\right\vert =\left\vert (\xi -\xi \,\cos
\theta )\int_{0}^{1}\big(\partial _{\xi }^{2}G_{\delta }\big)(\xi _{\tau
})d\tau \right\vert \\
&&\hskip3cm\leq C|\xi |\,\sin ^{2}(\theta /2)\,\,\langle \xi \rangle
^{2(2s^{\prime }-1)}\int_{0}^{1}G_{\delta }(\xi _{\tau })d\tau \\
&&\hskip3cm\leq C\,\sin ^{2}(\theta /2)\,\,\langle \xi \rangle ^{4s^{\prime
}-1}G_{\delta }(\xi \sin \theta )G_{\delta }(t,\xi \cos \theta ).
\end{eqnarray*}%
The case $\pi /4\leq |\theta |\leq \pi /2$ is similar to (\ref{2.9}). Thus,
we have proved Lemma\ref{lemm2.2}.
\end{proof}

We now study the commutators of Kac's collision operators with the
above mollifier operators.

\begin{proposition}
\label{prop2.1} Assume that $0<s^{\prime}<1/2$, Let $f, g\in L^2_1({{\mathbb{%
R}}}_v)$ and $h\in H^{s^{\prime}}({{\mathbb{R}}}_v)$, then we
have that
\begin{equation}  \label{2.12+0}
\begin{split}
&\left|\Big(G_\delta\,K(f, \, g),\,\, h\Big)- \Big(K(f, \,G_\delta\,g),\,\, h%
\Big)\right| \\
&\leq C\,\|{G_\delta}\,f\|_{L^2_1({{\mathbb{R}}})} \,\|{G_\delta}%
\,g\|_{H^{s^{\prime}}({{\mathbb{R}}})}\|h\|_{H^{s^{\prime}}({{\mathbb{R}}})},
\end{split}%
\end{equation}
and
\begin{equation}  \label{2.12}
\begin{split}
&\left|\Big((v\,G_\delta)\,K(f, \, g),\,\, h\Big)- \Big(K(f, \,
(v\,G_\delta)\,g),\,\, h\Big)\right| \\
&\leq C\,\left(\|f\|_{L^1_1({{\mathbb{R}}})}+\|{G_\delta}\,f\|_{L^2_1({{%
\mathbb{R}}})}\right) \,\|{G_\delta}\,g\|_{H^{s^{\prime}}_1({{\mathbb{R}}}%
)}\|h\|_{H^{s^{\prime}}({{\mathbb{R}}})}.
\end{split}%
\end{equation}
\end{proposition}
\begin{proof}
By definition, we have, for a suitable function $F$,
\begin{equation}
{\mathcal{F}}(G_{\delta }\,F)(\xi )=G_{\delta }(t,\xi )\hat{F}(\xi ),
\label{2.13+01}
\end{equation}%
and
\begin{equation}
{\mathcal{F}}((v\,G_{\delta })\,F)(\xi )=i\partial _{\xi }\big(G_{\delta }\,%
\hat{F}\big)(\xi )=i(\partial _{\xi }G_{\delta })(\xi )\,\hat{F}(\xi
)+iG_{\delta }(\xi )\,(\partial _{\xi }\hat{F})(\xi ).  \label{2.13+02}
\end{equation}%
By using the Bobylev formula (\ref{2.2}) and the Plancherel
formula,
\begin{equation*}
\begin{split}
& (2\pi )^{1/2}\left\{ \Big(G_{\delta }\,K(f,\,g),\,\,h\Big)-\Big(%
K(f,\,G_{\delta }\,g),\,\,h\Big)\right\} \\
=& \int_{{{\mathbb{R}}}_{\xi }}\int_{-\frac{\pi }{2}}^{\frac{\pi }{2}}\beta
(\theta )G_{\delta }(\xi )\Big\{\hat{f}(\xi \,\sin {\theta })\hat{g}(\xi
\,\cos {\theta })-\hat{f}(0)\hat{g}(\xi )\Big\}d\theta \,\overline{\hat{h}%
(\xi )}d\xi \\
& -\int_{{{\mathbb{R}}}_{\xi }}\int_{-\frac{\pi }{2}}^{\frac{\pi }{2}}\beta
(\theta )\Big\{\hat{f}(\xi \,\sin {\theta })\,\big({\mathcal{F}}(G_{\delta
}\,{g})\big)(\xi \,\cos {\theta })-\hat{f}(0)\big({\mathcal{F}}(G_{\delta }\,%
{g})\big)(\xi )\Big\}d\theta \,\,\overline{\hat{h}(\xi )}d\xi \\
=& \int_{{{\mathbb{R}}}_{\xi }}\int_{-\frac{\pi }{2}}^{\frac{\pi }{2}}\beta
(\theta )\hat{f}(\xi \,\sin {\theta })\Big\{G_{\delta }(\xi )-G_{\delta
}(\xi \,\cos {\theta })\Big\}\hat{g}(\xi \,\cos \theta )\,\,\overline{\hat{h}%
(\xi )}\,d\theta \,d\xi \,.
\end{split}%
\end{equation*}%
The above formula can be justified by the cutoff approximation of
collision kernel $\beta (\theta )$, then (\ref{2.9}) and
(\ref{1.1+2}) imply
\begin{equation*}
\begin{split}
& \left\vert \Big(G_{\delta }\,K(f,\,g),\,\,h\Big)-\Big(K(f,\,G_{\delta
}\,g),\,\,h\Big)\right\vert \\
\leq & C\int_{{{\mathbb{R}}}_{\xi }}\int_{-\frac{\pi }{2}}^{\frac{\pi }{2}%
}\beta (\theta )\sin ^{2}(\theta /2)|G_{\delta }(\xi \,\sin \theta )\,\hat{f}%
(\xi \,\sin {\theta })| \\
& \hskip3cm\times |G_{\delta }(\xi \,\cos {\theta })\,\hat{g}(\xi \,\cos
\theta )|\,\,\langle \xi \rangle ^{2s^{\prime }}\,|\hat{h}(\xi )|\,d\theta
\,d\xi \\
\leq & C\Vert |G_{\delta }\,\hat{f}\Vert _{L^{\infty }({{\mathbb{R}}}_{\xi
})}\int_{-\frac{\pi }{2}}^{\frac{\pi }{2}}\beta (\theta )\sin ^{2}(\theta /2)
\\
& \hskip2cm\times \,\left( \int_{{{\mathbb{R}}}_{\xi }}\,\langle \xi \rangle
^{2s^{\prime }}\,\left\vert G_{\delta }(\xi \,\cos {\theta })\hat{g}(\xi
\,\cos \theta )\right\vert ^{2}d\xi \right) ^{1/2}\,\,\Vert h\Vert
_{H^{s^{\prime }}({{\mathbb{R}}})}\,d\theta \\
\leq & C\Vert |G_{\delta }\,f\Vert _{L^{1}({{\mathbb{R}}}_{v})}\int_{-\frac{%
\pi }{2}}^{\frac{\pi }{2}}\beta (\theta )\frac{\sin ^{2}(\theta
/2)}{{|\cos \theta |^{1/2+s'}}}d\,\theta \,\Vert \,\langle \,\cdot
\,\rangle ^{s^{\prime }}\,G_{\delta }\hat{g}\Vert
_{L^{2}({{\mathbb{R}}}_{\xi })}\,\,\Vert h\Vert
_{H^{s^{\prime }}({{\mathbb{R}}})} \\
\leq & C\Vert |G_{\delta }\,f\Vert _{L_{1}^{2}({{\mathbb{R}}}_{v})}\Vert
G_{\delta }\,g\Vert _{H^{s^{\prime }}({{\mathbb{R}}})}\,\,\Vert h\Vert
_{H^{s^{\prime }}({{\mathbb{R}}})}\,,
\end{split}%
\end{equation*}%
where we have used the following continuous embedding
\begin{equation*}
L_{\alpha }^{2}({{\mathbb{R}}})\subset L^{1}({{\mathbb{R}}}),\quad \alpha
>1/2.
\end{equation*}%
We have proved (\ref{2.12+0}).

To treat (\ref{2.12}), by using (\ref{2.13+02}), we similarly have,
\begin{equation*}
\begin{split}
& (2\pi )^{1/2}\left\{ \Big((v\,G_{\delta })\,K(f,\,g),\,\,h\Big)-\Big(%
K(f,\,(v\,G_{\delta })\,g),\,\,h\Big)\right\} \\
=& \int_{{{\mathbb{R}}}_{\xi }}\int_{-\frac{\pi }{2}}^{\frac{\pi }{2}}\beta
(\theta )\Big\{i\partial _{\xi }\left( G_{\delta }(\xi )\hat{f}(\xi \,\sin {%
\theta })\hat{g}(\xi \,\cos {\theta })\right) \\
& \hskip3cm-\hat{f}(\xi \,\sin {\theta }){\mathcal{F}}\big((v\,G_{\delta
})\,g\big)(\xi \,\cos {\theta })\Big\}\,\,\overline{\hat{h}(\xi )}\,\,d\xi
\,d\theta \\
=& i\int_{{{\mathbb{R}}}_{\xi }}\int_{-\frac{\pi }{2}}^{\frac{\pi }{2}}\beta
(\theta )\,\sin \theta \,(\partial _{\xi }\hat{f})(\xi \,\sin {\theta }%
)G_{\delta }(\xi )\hat{g}(\xi \,\cos {\theta })\,\,\overline{\hat{h}(\xi )}%
\,\,d\xi \,d\theta \\
+& i\int_{{{\mathbb{R}}}_{\xi }}\int_{-\frac{\pi }{2}}^{\frac{\pi }{2}}\beta
(\theta )\hat{f}(\xi \,\sin {\theta })\Big\{\partial _{\xi }\big(G_{\delta
}(\xi )\hat{g}(\xi \,\cos {\theta })\big)-\big(\partial _{\xi }\,(G_{\delta
}\,\hat{g})\big)(\xi \,\cos {\theta })\Big\}\,\,\overline{\hat{h}(\xi )}%
\,\,d\xi \,d\theta \, \\
=& \,\,(I)+(II).
\end{split}%
\end{equation*}%
For the term $(I)$, we have
\begin{equation*}
\begin{split}
&|(I)|\leq  \int_{{{\mathbb{R}}}_{\xi }}\int_{-\frac{\pi }{2}}^{\frac{\pi }{2%
}}\beta (\theta )\,|\sin \theta |\,|(\partial _{\xi }\hat{f})(\xi \,\sin {%
\theta })|\,\,\big|G_{\delta }(\xi \,\cos \theta )\hat{g}(\xi \,\cos {\theta
})\big|\,\,\big|\overline{\hat{h}(\xi )}\big|\,\,d\xi \,d\theta \\
+& \int_{{{\mathbb{R}}}_{\xi }}\int_{-\frac{\pi }{2}}^{\frac{\pi
}{2}}\beta
(\theta )\,|\sin \theta |\,|(\partial _{\xi }\hat{f})(\xi \,\sin {\theta }%
)|\,\,\Big|G_{\delta }(\xi )-G_{\delta }(\xi \,\cos \theta )\Big||\hat{g}%
(\xi \,\cos {\theta })|\,\,\big|\overline{\hat{h}(\xi )}\big|\,\,d\xi
\,d\theta \\
&\leq  \,\,I_{1}+I_{2}.
\end{split}%
\end{equation*}%
Firstly, (\ref{1.1+2}) with the hypothesis $0<s<1/2$ implies that
\begin{equation*}
\begin{split}
I_{1}& \leq \Vert \partial _{\xi }\hat{f}\Vert _{L^{\infty }({{\mathbb{R}}}%
_{\xi })}\int_{{{\mathbb{R}}}_{\xi }}\int_{-\frac{\pi }{2}}^{\frac{\pi }{2}%
}\beta (\theta )|\sin \theta |\,\,\big|G_{\delta }(\xi \,\cos {\theta })\,%
\hat{g}(\xi \,\cos {\theta })\big|\,\,\big|\hat{h}(\xi )\big|\,\,d\xi
\,d\theta \\
& \leq C\Vert f\Vert _{L_{1}^{1}({{\mathbb{R}}}_{v})}\int_{-\frac{\pi }{2}}^{%
\frac{\pi }{2}}\beta (\theta )\frac{|\sin \theta |}{{|\cos \theta |^{1/2}}}%
\,d\theta \,\,\,\Vert \hat{h}\Vert _{L^{2}({{\mathbb{R}}}_{\xi })} \\
& \hskip2cm\times \Big(\int_{{{\mathbb{R}}}_{\xi }}\big|G_{\delta }(\xi
\,\cos {\theta })\,\hat{g}(\xi \,\cos {\theta })\big|^{2}d(\xi \,\cos {%
\theta })\Big)^{1/2}\,\, \\
& \leq C\Vert f\Vert _{L_{1}^{1}({{\mathbb{R}}}_{v})}\Vert G_{\delta }\,{g}%
\Vert _{L^{2}({{\mathbb{R}}}_{v})}\,\Vert h\Vert _{L^{2}({{\mathbb{R}}}%
_{v})}\,.
\end{split}
\end{equation*}
For the term $I_{2}$, by using (\ref{2.9}) we have the following estimates
which are also true for $0<s<1$),
\begin{equation*}
\begin{split}
I_{2}& \leq \int_{{{\mathbb{R}}}_{\xi }}\int_{-\frac{\pi }{2}}^{\frac{\pi }{2%
}}\beta (\theta )|\sin \theta |\,\sin ^{2}(\theta /2)\,\,|G_{\delta }(\xi
\,\sin {\theta })(\partial _{\xi }\hat{f})(\xi \,\sin {\theta })| \\
& \hskip2cm\times \,\,\big|G_{\delta }(\xi \,\cos {\theta })\,\hat{g}(\xi
\,\cos {\theta })\big|\,\,\langle \xi \rangle ^{2s^{\prime }}\,\big|\hat{h}%
(\xi )\big|\,\,d\xi \,d\theta \\
& \leq C\Vert \langle \,\cdot \,\rangle ^{s^{\prime }}G_{\delta }\,\hat{g}%
\Vert _{L^{\infty }({{\mathbb{R}}}_{\xi })}\int_{-\frac{\pi }{2}}^{\frac{\pi
}{2}}\beta (\theta )\frac{|\sin \theta |\,\sin ^{2}(\theta /2)}{|\sin \theta
|^{1/2}}\,d\theta \,\,\Vert h\Vert _{H^{s^{\prime }}({{\mathbb{R}}}_{v})} \\
& \hskip2cm\times \Big(\int_{{{\mathbb{R}}}_{\xi }}\big|G_{\delta }(\xi
\,\sin {\theta })\,(\partial _{\xi }\hat{f})(\xi \,\sin {\theta })\big|%
^{2}d(\xi \,\sin {\theta })\Big)^{1/2}\,\, \\
& \leq C\Vert \langle D_{v}\rangle ^{s^{\prime }}G_{\delta }\,{g}\Vert
_{L^{1}({{\mathbb{R}}}_{v})}\int_{-\frac{\pi }{2}}^{\frac{\pi }{2}}\beta
(\theta )\frac{|\sin \theta |\,\sin ^{2}(\theta /2)}{|\sin \theta |^{1/2}}%
\,d\theta \,\,\Vert h\Vert _{H^{s^{\prime }}({{\mathbb{R}}}_{v})} \\
& \hskip2cm\times \Big(\int_{{{\mathbb{R}}}_{\xi }}\big|G_{\delta }(\xi
\,\sin {\theta })\,(\partial _{\xi }\hat{f})(\xi \,\sin {\theta })\big|%
^{2}d(\xi \,\sin {\theta })\Big)^{1/2}\,\, \\
& \leq C\Vert G_{\delta }\,{g}\Vert _{H_{1}^{s^{\prime }}({{\mathbb{R}}}%
_{v})}\Vert G_{\delta }\,(v\,f)\Vert _{L^{2}({{\mathbb{R}}}_{v})}\,\Vert
h\Vert _{H^{s^{\prime }}({{\mathbb{R}}}_{v})}\,.
\end{split}%
\end{equation*}%
Moreover, for a suitable function $F$, we have
\begin{equation*}
G_{\delta }\,(v\,F)=v\,G_{\delta }\,F+[G_{\delta },\,\,v]\,F,
\end{equation*}%
and
\begin{equation*}
{\mathcal{F}}\big([G_{\delta },\,\,v]\,F\big)(\xi )=i(\partial _{\xi
}G_{\delta })(\xi )\hat{F}(\xi ).
\end{equation*}%
Then the symbolic calculus (\ref{2.7}) implies that, for $0<2s^{\prime }<1$,
we have
\begin{equation}
\Vert G_{\delta }\,(v\,F)\Vert _{H^{\alpha }({{\mathbb{R}}}_{v})}\leq C\Vert
G_{\delta }\,F\Vert _{H_{1}^{\alpha }({{\mathbb{R}}}_{v})}\,  \label{2.13}
\end{equation}%
for any $\alpha \geq 0$, then
\begin{equation}
|(I)|\leq C\left\{ \Vert f\Vert _{L_{1}^{1}({{\mathbb{R}}}_{v})}+\Vert
G_{\delta }\,f\Vert _{L_{1}^{2}({{\mathbb{R}}}_{v})}\right\} \Vert G_{\delta
}\,g\Vert _{H_{1}^{s^{\prime }}({{\mathbb{R}}}_{v})}\,\Vert h\Vert
_{H^{s^{\prime }}({{\mathbb{R}}}_{v})}  \label{2.14}
\end{equation}%
On the other hand, for the term $(II)$, we have
\begin{equation*}
\begin{split}
& \partial _{\xi }\big(G_{\delta }(\xi )\hat{g}(\xi \,\cos {\theta })\big)-%
\big(\partial _{\xi }\,(G_{\delta }\,\hat{g})\big)(\xi \,\cos {\theta })=%
\big\{G_{\delta }(\xi )-G_{\delta }(\xi \,\cos {\theta })\big\}\big(\partial
_{\xi }\hat{g}\big)(\xi \,\cos {\theta }) \\
& +G_{\delta }(\xi )\big(\cos \theta -1\big)(\partial _{\xi }\hat{g})(\xi
\,\cos {\theta })+\big\{(\partial _{\xi }G_{\delta })(\xi )-(\partial _{\xi
}G_{\delta })(\xi \,\cos {\theta })\big\}\,\,\hat{g}(\xi \,\cos {\theta }) \\
& =A_{1}+A_{2}+A_{3}.
\end{split}%
\end{equation*}%
Thus
\begin{equation*}
|(II)|\leq C\int_{{{\mathbb{R}}}_{\xi }}\int_{-\frac{\pi }{2}}^{\frac{\pi }{2%
}}\beta (\theta )|\hat{f}(\xi \,\sin {\theta })|\big|A_{1}+A_{2}+A_{3}\big|%
\,\,\big|\hat{h}(\xi )\big|\,\,d\xi \,d\theta \,.
\end{equation*}%
We study now the above 3 terms on the right-hand side. By using (\ref{2.9}),
\begin{equation*}
\begin{split}
& \int_{{{\mathbb{R}}}_{\xi }}\int_{-\frac{\pi }{2}}^{\frac{\pi }{2}}\beta
(\theta )|\hat{f}(\xi \,\sin {\theta })|\,|A_{1}|\,\,\big|\hat{h}(\xi )\big|%
\,\,d\xi \,d\theta \\
& \leq \int_{{{\mathbb{R}}}_{\xi }}\int_{-\frac{\pi }{2}}^{\frac{\pi }{2}%
}\beta (\theta )\,\sin ^{2}(\theta /2)\,\,|G_{\delta }(\xi \,\sin {\theta })%
\hat{f}(\xi \,\sin {\theta })| \\
& \hskip2cm\times \,\,\big|G_{\delta }(\xi \,\cos {\theta })\,(\partial
_{\xi }\hat{g})(\xi \,\cos {\theta })\big|\,\,\langle \xi \rangle
^{2s^{\prime }}\,\big|\hat{h}(\xi )\big|\,\,d\xi \,d\theta \\
& \leq C\Vert G_{\delta }\,{f}\Vert _{L^{1}({{\mathbb{R}}}_{v})}\int_{-\frac{%
\pi }{2}}^{\frac{\pi }{2}}\beta (\theta )\frac{\sin ^{2}(\theta /2)}{|\cos
\theta |^{1/2}}\,d\theta \,\Vert G_{\delta }\,(v\,g)\Vert _{H^{s^{\prime }}({%
{\mathbb{R}}}_{v})}\,\Vert h\Vert _{H^{s^{\prime }}({{\mathbb{R}}}_{v})} \\
& \leq C\Vert G_{\delta }\,f\Vert _{L_{1}^{2}({{\mathbb{R}}}_{v})}\Vert
G_{\delta }\,(v\,g)\Vert _{H^{s^{\prime }}({{\mathbb{R}}}_{v})}\,\Vert
h\Vert _{H^{s^{\prime }}({{\mathbb{R}}}_{v})}\,.
\end{split}%
\end{equation*}%
The estimate (\ref{2.11}) and $\cos \theta -1=-2\sin ^{2}(\theta /2)$ imply
\begin{equation*}
\begin{split}
& \int_{{{\mathbb{R}}}_{\xi }}\int_{-\frac{\pi }{2}}^{\frac{\pi }{2}}\beta
(\theta )|\hat{f}(\xi \,\sin {\theta })|\,|A_{2}|\,\,\big|\hat{h}(\xi )\big|%
\,\,d\xi \,d\theta \\
& \leq C\int_{{{\mathbb{R}}}_{\xi }}\int_{-\frac{\pi }{2}}^{\frac{\pi }{2}%
}\beta (\theta )\,\sin ^{2}(\theta /2)\,\,|G_{\delta }(\xi \,\sin {\theta })%
\hat{f}(\xi \,\sin {\theta })| \\
& \hskip3cm\times \,\,\big|G_{\delta }(\xi \,\cos {\theta })\,(\partial
_{\xi }\hat{g})(\xi \,\cos {\theta })\big|\,\big|\hat{h}(\xi )\big|\,\,d\xi
\,d\theta \\
& \leq C\Vert G_{\delta }\,f\Vert _{L_{1}^{2}({{\mathbb{R}}}_{v})}\Vert
G_{\delta }\,(v\,g)\Vert _{L^{2}({{\mathbb{R}}}_{v})}\,\Vert h\Vert _{L^{2}({%
{\mathbb{R}}}_{v})}\,.
\end{split}%
\end{equation*}%
Finally, the hypothesis $0<s<1/2$ implies $(4s^{\prime }-1)^{+}<2s^{\prime }$%
, then (\ref{2.10}) yields,
\begin{equation*}
\begin{split}
& \int_{{{\mathbb{R}}}_{\xi }}\int_{-\frac{\pi }{2}}^{\frac{\pi }{2}}\beta
(\theta )|\hat{f}(\xi \,\sin {\theta })|\,|A_{3}|\,\,\big|\hat{h}(\xi )\big|%
\,\,d\xi \,d\theta \\
& \leq C\int_{{{\mathbb{R}}}_{\xi }}\int_{-\frac{\pi }{2}}^{\frac{\pi }{2}%
}\beta (\theta )\,\sin ^{2}(\theta /2)\,\,|G_{\delta }(\xi \,\sin {\theta })%
\hat{f}(\xi \,\sin {\theta })| \\
& \hskip3cm\times \,\,\,\langle \xi \rangle ^{2s^{\prime }}\,\,\big|%
G_{\delta }(\xi \,\cos {\theta })\,\hat{g}(\xi \,\cos {\theta })\big|\,\,%
\big|\hat{h}\big|\,\,d\xi \,d\theta \\
& \leq C\Vert G_{\delta }\,f\Vert _{L_{1}^{2}({{\mathbb{R}}}_{v})}\Vert
G_{\delta }\,g\Vert _{H^{s^{\prime }}({{\mathbb{R}}}_{v})}\,\Vert h\Vert
_{H^{s^{\prime }}({{\mathbb{R}}}_{v})}\,.
\end{split}%
\end{equation*}%
By summing the above 3 estimates, (\ref{2.13}) implies that
\begin{equation}
|(II)|\leq C\Vert G_{\delta }\,f\Vert _{L_{1}^{2}({{\mathbb{R}}}_{v})}\Vert
G_{\delta }\,g\Vert _{H_{1}^{s^{\prime }}({{\mathbb{R}}}_{v})}\,\Vert h\Vert
_{H^{s^{\prime }}({{\mathbb{R}}}_{v})}\,.  \label{2.13+b}
\end{equation}%
Proof of Proposition \ref{prop2.1} is established.
\end{proof}

\begin{remark}
In the proof of estimate for the term $I_1$ and the last term of $(II)$, we
have used crucially the restrict assumption $0<s<1/2$.
\end{remark}

%%%%%%%%%%%%%%%%%%%%%%%%%%%%%%%%%%%%%%%%%%%%%%%%%%%%%%%%%%%%%%%%%%%%%%%
%%%%%%%%%%%%%%%%%%%%%%%%%%%%%%%%%%%%%%%%%%%%%%%%%%%%%%%%%%%%%%%%%%%%%%%
\vskip0.5cm

\section{Sobolev regularizing effect of weak solutions}

\label{section3} \smallbreak

We will first give an  $H^{+\infty }$-regularizing effect results
for Kac's
equation. The following Theorem is more precise than Theorem 1.1 of \cite%
{MUXY-DCDS} where the homogeneous Boltzmann equation with Maxwellian
molecules has been studied.

\begin{theorem}
\label{theo2.1} Assume that the initial datum $f_{0}\in
L_{2+2s}^{1}\cap L\log L({{\mathbb{R}}})$, and the cross-section
$\beta $  satisfy (\ref{1.1+1} ) with $0<s<\frac 1 2$. If $f\in
L^{\infty }(]0,+\infty \lbrack ;L_{2+2s}^{1}\cap L\log
L({{\mathbb{R}}}))$ is a nonnegative weak solution of the Cauchy
problem (\ref{1.1}), then $f(t,\cdot )\in H_{2}^{+\infty
}({{\mathbb{R}}})$ for any $t>0$.
\end{theorem}

\begin{remark}
{\bf 1)} This is a $H^{+\infty }$-smoothing effect results for the
Cauchy problem, it is different from that of
\cite{D95,graham-meleard} where their assumption is that all moments
of the initial datum are bounded.

{\bf 2)} The results of theorem \ref{theo2.1} is also true if we assume the following
Debye-Yukawa type collision kernel :
\begin{equation*}
\beta (\theta )=C_{0}\frac{|\cos \theta |}{|\sin \theta |}\Big(\log |\theta
|^{-1}\Big)^{m},\,\,\quad \,\,0<m.
\end{equation*}
\end{remark}

To prove the Theorem \ref{theo2.1}, we use, as in \cite{MUXY-DCDS},
the mollifier of polynomial type
\begin{equation*}
M_{\delta }(t,\,\xi )=\langle \xi \rangle ^{tN-1}(1+\delta |\xi
|^{2})^{-N_{0}},
\end{equation*}%
for $0<\delta <1,\,t\in \lbrack 0,T_{0}]$ and $2N_{0}=T_{0}N+4$.

The idea is the same as the section 3 of \cite{MUXY-DCDS}, but now
we need to estimate the commutators with weighted $\langle
v\rangle^2$. It is analogous to the computation of preceding section.
We give here only the main points of the proof,
\begin{lemma}
\label{lemm2+1.1} We have that for any $0<\delta<1$ and $%
0\leq t\leq T_0,\, \xi\in{{\mathbb{R}}}$,
\begin{equation*}
\left|\partial_t M_\delta(t, \xi)\right|\leq N\log \big(\langle
\xi\rangle\big)  M_\delta(t, \xi) .
\end{equation*}
For $-\pi/4\leq\theta\leq\pi/4$,
\begin{equation*}
|M_{\delta}(\xi)-M_{\delta}(\xi \cos\theta)|\leq C
\sin^{2}(\theta/2) M_{\delta}(\xi\cos\theta)\,,
\end{equation*}
\begin{equation*}
\left|\big(\partial_\xi M_{\delta}\big)(\xi)-\big(\partial_\xi M_{\delta}%
\big)(\xi \cos\theta)\right|\leq C \sin^{2}(\theta/2)\langle\xi
\rangle^{-1} M_{\delta}(\xi\cos\theta),
\end{equation*}
and
\begin{equation*}
\left|\big(\partial^2_\xi M_{\delta}\big)(\xi)-\big(\partial^2_\xi M_{\delta}%
\big)(\xi \cos\theta)\right|\leq C \sin^{2}(\theta/2)\langle\xi
\rangle^{-2} M_{\delta}(\xi\cos\theta),
\end{equation*}
where the constant $C$ depends on $T_0, N$, but is independents of
 $0<\delta<1$.
\end{lemma}
We prove also this Lemma by using the Taylor formula, and for
any $k\in {\mathbb{N}}$,
$$
\left|\partial^k_\xi M_{\delta}(\xi)\right|\leq C_k
\langle\xi \rangle^{-k} M_{\delta}(\xi),\quad \xi\in{{\mathbb{R}}}
$$
with $C_k$ depends on $T_0, N$, but is independents of
$0<\delta<1$. Moreover, for the polynomial mollifier, we can
substitute the inequality (\ref{2.11}) by the following inequality,
\begin{equation}\label{2+1.11}
M_{\delta}(\xi)\leq C M_{\delta}(\xi\cos\theta)\,,\quad
-\frac{\pi}{4}\leq\theta\leq\frac{\pi}{4},
\end{equation}
here again $C$ depending on  $N_0, T$, and independents of
  $\delta>0$. We have therefore
\begin{proposition}
\label{prop2+1.1} Assume that $0<s<1/2$, we have that
\begin{equation}\label{2+1.12}
\begin{split}
&\left|\Big(\big(v\,M_\delta\big)\,K(f, \, g),\,\, h\Big)- \Big(K(f,
\,
(v\, M_\delta)\,g),\,\, h\Big)\right| \\
&\leq C\,\|f\|_{L^1_1({{\mathbb{R}}})} \,\|{M_\delta}\,g\|_{L^{2}_1({{\mathbb{R}}}%
)}\|h\|_{L^{2}({{\mathbb{R}}})},
\end{split}%
\end{equation}
and
\begin{equation}\label{2+1.13}
\begin{split}
&\left|\Big(\big(\langle v\rangle^2\,M_\delta\big)\,K(f, \, g),\,\,
h\Big)- \Big(K(f, \,
(\langle v\rangle^2\, M_\delta)\,g),\,\, h\Big)\right| \\
&\leq C\,\|f\|_{L^1_2({{\mathbb{R}}})} \,\|{M_\delta}\,g\|_{L^{2}_2({{\mathbb{R}}}%
)}\|h\|_{L^{2}({{\mathbb{R}}})},
\end{split}%
\end{equation}
\end{proposition}
The proof of (\ref{2+1.12})  is similar to (\ref{2.12}) where we
substitute Lemma \ref{lemm2.2} by Lemma \ref{lemm2+1.1}, and replace
(\ref{2.11}) by (\ref{2+1.11}). Consider now the estimate
(\ref{2+1.13}), we have, as in the proof of the proposition
\ref{prop2.1},
\begin{equation*}
\begin{split}
& (2\pi )^{1/2}\left\{ \Big((v^2\,M_{\delta })\,K(f,\,g),\,\,h\Big)-\Big(%
K(f,\,(v^2\,M_{\delta })\,g),\,\,h\Big)\right\} \\
=& -\int_{{{\mathbb{R}}}_{\xi }}\int_{-\frac{\pi }{2}}^{\frac{\pi
}{2}}\beta
(\theta )\,\sin^2 \theta \,(\partial^2_{\xi }\hat{f})(\xi \,\sin {\theta }%
)M_{\delta }(\xi )\hat{g}(\xi \,\cos {\theta })\,\,\overline{\hat{h}(\xi )}%
\,\,d\xi \,d\theta
\\
-& 2\int_{{{\mathbb{R}}}_{\xi }}\int_{-\frac{\pi }{2}}^{\frac{\pi
}{2}}\beta
(\theta )\,\sin \theta \,(\partial_{\xi }\hat{f})(\xi \,\sin {\theta }%
)\left(\partial_\xi \big(M_{\delta }(\xi )\hat{g}(\xi \,\cos {\theta })
\big)\right)\,\,\overline{\hat{h}(\xi )}%
\,\,d\xi \,d\theta\\
-& \int_{{{\mathbb{R}}}_{\xi }}\int_{-\frac{\pi }{2}}^{\frac{\pi
}{2}}\beta (\theta )\hat{f}(\xi \,\sin {\theta })\big\{\partial^2
_{\xi }\big(M_{\delta }(\xi )\hat{g}(\xi \,\cos {\theta
})\big)-\big(\partial^2_{\xi }\,(M_{\delta }\,\hat{g})\big)(\xi
\,\cos {\theta })\big\}\,\,\overline{\hat{h}(\xi )}
\,\,d\xi \,d\theta \, \\
=& \,\,B_1+B_2+B_3.
\end{split}
\end{equation*}
Then
$$
|B_1|\leq C\,\|\partial^2_{\xi }\hat{f}\|_{L^\infty({{\mathbb{R}}})}
\,\|{M_\delta}\,g\|_{L^{2}({{\mathbb{R}}}
)}\|h\|_{L^{2}({{\mathbb{R}}})}\leq C\,\|f\|_{L^1_2({{\mathbb{R}}})}
\,\|{M_\delta}\,g\|_{L^{2}({{\mathbb{R}}}
)}\|h\|_{L^{2}({{\mathbb{R}}})},
$$
and for $0<2s<1$,
$$
|B_2|\leq C\,\|f\|_{L^1_1({{\mathbb{R}}})} \,\left(\|{M_\delta}\,g\|_{L^{2}({{\mathbb{R}}}
)}+\|{M_\delta}\,(v\, g)\|_{L^{2}({{\mathbb{R}}}
)}\right)\|h\|_{L^{2}({{\mathbb{R}}})}.
$$
The term $B_3$ is evidently more complicate, but the idea is the same,
we omit here their computations.

\bigbreak
Using the continuous embedding
\begin{equation*}
L_{\ell }^{1}({{\mathbb{R}}})\subset H_{\ell }^{-1}({{\mathbb{R}}}),
\end{equation*}%
the upper bounded (\ref{estimate-E}) with $m=-2,
\ell=2$ and $0<2s<1$ imply ,
\begin{equation*}
\Vert K(g,\,h)\Vert _{H_{2}^{-2}({{\mathbb{R}}}_{v})}\leq C\Vert
g\Vert _{L_{2+2s}^{1}({{\mathbb{R}}}_{v})}\Vert h\Vert _{H_{2
+2s}^{-2+2s}({{\mathbb{R}}}_{v})}\leq  C\Vert g\Vert_{L_{2+2s}^{1}({{\mathbb{R}}}_{v})}\,
\Vert h\Vert_{L_{2+2s}^{1}({{\mathbb{R}}}_{v})}\,.
\end{equation*}%
Let $f\in L^{\infty }(]0,+\infty \lbrack
;L_{2+2s}^{1}({{\mathbb{R}}}))$ be a weak solution of the Cauchy
problem (\ref{1.1}), then we can take
\begin{equation*}
f_1=M_{\delta }(t,D_{v})\langle v\rangle ^{4}M_{\delta
}(t,D_{v})f\in L^{\infty
}([0,T_{0}];H_{-2+2s}^{5}({{\mathbb{R}}}))\, ,
\end{equation*}%
as test functions of the Cauchy problem (\ref{1.1}). By using
similar manipulations as in \cite{MUXY-DCDS}, we can obtain the
regularity with respect to $t$ variable, to simplify the notations
we suppose that $f_1\in
C^{1}([0,T_{0}];H_{-2+2s}^{5}({{\mathbb{R}}}))$. We have
$$
\Big(\partial _{t}f(t,\,\cdot \,),\,\,f_1(t,\,\cdot
\,)\Big)_{L^{2}({{
\mathbb{R}}}_{v})}=\Big(K(f\,,\,\,f),\,\,f_1\Big)_{L^{2}({{\mathbb{R}
}}_{v})}.
$$
Then Lemma \ref{lemm2+1.1}, Proposition \ref{prop2+1.1}, the
coercivity estimate (\ref{2.1}) and the conservations (\ref{1.3}),
(\ref{1.3b}), (\ref{1.3c}) imply that
\begin{equation*}
\begin{split}
& \frac{d}{dt}\Vert M_\delta f(t)\Vert^2_{L_2^2({\mathbb{R}}
_{v})}+c_{f_0}\Vert M_\delta f(t)\Vert^2_{H_2^s({\mathbb{R}}
_v)} \\
\leq & C_{f_0}\Vert \log^{1/2} (|D_v|) M_\delta f(t)\Vert^2
_{L_2^2({\mathbb{R}}_v)}+C\,\Vert f_0\Vert _{L_2^1({\mathbb{R}}
)}\,\Vert M_\delta\,f(t)\Vert^2_{L_2^2({\mathbb{R}})}.
\end{split}
\end{equation*}
We now use the following interpolation inequality, for any small
$\varepsilon >0$
\begin{equation}
\Vert \log^{1/2} (|D_v|) M_\delta f(t)\Vert^2
_{L_2^2({\mathbb{R}}_v)}\leq \varepsilon  \Vert M_\delta
f(t)\Vert^2_{H_2^s({\mathbb{R}} _v)}+C_\varepsilon \Vert M_\delta
f(t)\Vert^2_{L_2^2({\mathbb{R}} _v)}.
\end{equation}
Then for $t\in [0, T_0]$,
\begin{equation*}
\frac{d}{dt}\Vert M_\delta f(t)\Vert^2_{L_2^2({\mathbb{R}}
_{v})}\leq C_1\,\Vert M_\delta\,f(t)\Vert^2_{L_2^2({\mathbb{R}})}
\end{equation*}
where $C_1$ depends on $T_0, N$, but  independents of 
$0<\delta<1$. So that for $t\in [0, T_0]$,
\begin{equation*}
\Vert M_\delta f(t)\Vert_{L_2^2({\mathbb{R}} _{v})}\leq
e^{C_1\,t}\,\Vert M_\delta\,f(0)\Vert_{L_2^2({\mathbb{R}})} \leq
e^{C_1\,t}\,\Vert f_0\Vert_{H_2^{-1}({\mathbb{R}})} \leq
e^{C_1\,t}\,\Vert f_0\Vert_{L_2^{1}({\mathbb{R}})}.
\end{equation*}
We have therefore proved for $t\in [0, T_0]$,
$$
(1+|D_v|^2)^{tN-1} f(t, \,\cdot\,)\in L^2_2({\mathbb{R}}).
$$
Since we can choose arbitrary $N>0$ and $T_0>0$, we have proved
Theorem \ref{theo2.1}.

%%%%%%%%%%%%%%%%%%%%%%%%%%%%%%%%%%%%%%%%%%%%%%%%%%%%%%%%%%%%%%%%%%%%%%%
%%%%%%%%%%%%%%%%%%%%%%%%%%%%%%%%%%%%%%%%%%%%%%%%%%%%%%%%%%%%%%%%%%%%%%%
\vskip0.5cm

\section{Gevrey regularizing effect of solutions}

\label{section3+1} \smallbreak

\smallbreak Theorem \ref{theo2.1} implies that the weak solution of the
Cauchy problem (\ref{1.1}) satisfies $f\in L^{\infty
}([t_{0},\,T_{0}[;H_{2}^{1}({{\mathbb{R}}}))$ for any $t_{0}>0$. Then $f$ is
a solution of the following Cauchy problem :
\begin{equation*}
\left\{
\begin{array}{ll}
\frac{\partial f}{\partial t}=K(f,f), & v\in {{\mathbb{R}}},\,\,t>t_{0}, \\
f|_{t=t_{0}}=f(t_{0},\,\cdot \,)\in H_{2}^{1}({{\mathbb{R}}}). &
\end{array}%
\right.
\end{equation*}

We now study the local Gevrey  regularizing effect of {  the} Cauchy
problem, and
suppose that the initial datum is $f_{0}\in H_{2}^{1}\cap L_{2}^{1}({{%
\mathbb{R}}})$. We state this result as the:

\begin{theorem}
\label{theo3.1} Assume that the initial datum $f_{0}\in H_{2}^{1}\cap
L_{2}^{1}({{\mathbb{R}}})$, and the cross-section $\beta $ satisfy (\ref%
{1.1+1}) with $0<s<\frac{1}{2}$. For $T_{0}>0$, if $f\in L^{\infty
}([0,T_{0}];H_{2}^{1}\cap L_{2}^{1}({{\mathbb{R}}}))$ is a
nonnegative weak solution of the Cauchy problem (\ref{1.1}), then
for any $0<s^{\prime }<s$, there exists $0<T_{\ast }\leq T_{0}$ such
that $f(t,\cdot )\in G^{^{\frac{1}{2s^{\prime }}}}({{\mathbb{R}}})$
for any $0<t\leq T_{\ast }$. More precisely, there exists $c_0>0$,
\begin{equation*}
e^{c_0 t\langle D_{v}\rangle ^{2s^{\prime }}}f\in L^{\infty
}([0,T_{\ast }];\,L^{2}_1({{\mathbb{R}}})).
\end{equation*}
\end{theorem}

\begin{remark}\label{rema3.2}
The above Gevrey smoothing effect property of Cauchy problem is for any weak solution $f\in
L^{\infty }([0,T_{0}];H_{2}^{1}\cap L_{2}^{1}({{\mathbb{R}}}))$, so
that we don't need to use the uniqueness of solution for Kac's
equation.
\end{remark}

We prove the above theorem by construction of \emph{a priori}
estimates for the{  mollified weak solution}. Take $f\in L^{\infty
}(]0,T_{0}[;H_{2}^{1}\cap L_{2}^{1}({{\mathbb{R}}}))$ to be a weak
solution of the Cauchy problem (\ref{1.1}), then (\ref{estimate-E})
with $m=\ell =0$ implies that, (recall the assumption $0<s<1/2$)
\begin{equation*}
K(f,\,f)\in L^{\infty }(]0,T_{0}[;L^{2}({{\mathbb{R}}}_{v})).
\end{equation*}%
So that we need to choose a test function $\varphi \in
C^{1}([0,T_{0}];\,L^{2}({{\mathbb{R}}}_{v}))$ to make sense
\begin{equation*}
\big(K(f,\,f),\,\,\varphi \big)_{L^{2}({{\mathbb{R}}}_{v})}.
\end{equation*}%
The right way is to choose  a mollified weak solution $f$, we
first have
\begin{equation*}
\tilde{f}(t,\,\cdot \,)=\Big(G_{\delta }(t,D_{v})\langle v\rangle
^{2}\,G_{\delta }(t,D_{v})f\Big)(t,\,\cdot \,)\in L^{\infty
}(]0,T_{0}[;\,\,H^{1}({{\mathbb{R}}})).
\end{equation*}%
Here again we suppose that $%
\tilde{f}\in C^{1}([0,T_{0}];\,H^{1}({{\mathbb{R}}}_{v}))$, and study the
equation of (\ref{1.1}) in the following weak formulation
\begin{equation}
\Big(\partial _{t}f(t,\,\cdot \,),\,\,\tilde{f}(t,\,\cdot \,)\Big)_{L^{2}({{%
\mathbb{R}}}_{v})}=\Big(K(f\,,\,\,f),\,\,\,\tilde{f}\Big)_{L^{2}({{\mathbb{R}%
}}_{v})}.  \label{3.1}
\end{equation}%
First, the left hand side term is
\begin{equation*}
\begin{split}
\Big(\partial _{t}f(t,\,\cdot \,),& \,\,\tilde{f}(t,\,\cdot \,)\Big)_{L^{2}({%
{\mathbb{R}}}_{v})}=\frac{1}{2}\frac{d}{dt}\Vert G_{\delta }f(t)\Vert
_{L_{1}^{2}({{\mathbb{R}}}_{v})}^{2} \\
& -\Big(\big(\partial _{t}G_{\delta }\big)(t,D_{v})f(t,\,\cdot
\,),\,\,G_{\delta }(t,D_{v})f(t,\,\cdot \,)\Big)_{L^{2}({{\mathbb{R}}}_{v})}
\\
& -\Big(v\,\big(\partial _{t}G_{\delta }\big)(t,D_{v})f(t,\,\cdot
\,),\,\,v\,G_{\delta }(t,D_{v})f(t,\,\cdot \,)\Big)_{L^{2}({{\mathbb{R}}}%
_{v})}.
\end{split}%
\end{equation*}%
Then we estimate the two terms on right hand side by using the following
lemma.

\begin{lemma}
\label{lemm3.1} There exists $C>0$ such that
\begin{equation}  \label{3.2}
\left|\Big(\big(\partial_t G_\delta\big)(t, D_v) f(t, \,\cdot\,),\, \,
G_\delta(t, D_v) f (t, \, \cdot\,)\Big)_{L^2({{\mathbb{R}}}_v)}\right|\leq\,
C \|G_\delta\, f \|^2_{H^{s^{\prime}}({{\mathbb{R}}}_v)},
\end{equation}
and
\begin{equation}  \label{3.3}
\left|\Big( v \,\big(\partial_t G_\delta\big)(t, D_v) f(t, \,\cdot\,),\,\,
v\, G_\delta(t, D_v) f (t, \, \cdot\,)\Big)_{L^2({{\mathbb{R}}}%
_v)}\right|\leq \, C\|G_\delta\, f\|^2_{H^{s^{\prime}}_1({{\mathbb{R}}}_v)}.
\end{equation}
\end{lemma}

\begin{proof}
(\ref{3.2}) can be deduced directly from (\ref{2.6}) by using the Plancherel
formula.

For (\ref{3.3}), we have
\begin{equation*}
\begin{split}
& \left\vert \Big(v\,\big(\partial _{t}G_{\delta }\big)(t,D_{v})f(t,\,\cdot
\,),\,\,v\,G_{\delta }(t,D_{v})f(t,\,\cdot \,)\Big)_{L^{2}({{\mathbb{R}}}%
_{v})}\right\vert \\
=& C\left\vert \int_{{{\mathbb{R}}}}\left( \partial _{\xi }\Big(c_{0}\langle
\xi \rangle ^{2s^{\prime }}G_{\delta }(t,\xi )\frac{1}{1+\delta
e^{c_{0}t\langle \xi \rangle ^{2s^{\prime }}}}\hat{f}(t,\xi )\Big)\right) \,%
\overline{{\mathcal{F}}\big(v\,G_{\delta }f\big)(t,\,\xi \,)}d\xi \right\vert
\\
\leq & \,C\int_{{{\mathbb{R}}}}\langle \xi \rangle ^{2s^{\prime }}\left\vert
\partial _{\xi }\Big(G_{\delta }(t,\xi )\hat{f}(t,\xi )\Big)\right\vert
\,\left\vert {{\mathcal{F}}\big(v\,G_{\delta }f\big)(t,\,\xi \,)}\right\vert
d\xi \\
& +\,C\int_{{{\mathbb{R}}}}\left\vert \partial _{\xi }\Big(\langle \xi
\rangle ^{2s^{\prime }}\frac{1}{1+\delta e^{c_{0}t\langle \xi \rangle
^{2s^{\prime }}}}\Big)\right\vert \,\,\big|G_{\delta }(t,\xi )\hat{f}(t,\xi )%
\big|\,\left\vert {{\mathcal{F}}\big(v\,G_{\delta }f\big)(t,\,\xi \,)}%
\right\vert d\xi \\
\leq & \,C\Vert G_{\delta }\,f\Vert _{H_{1}^{s^{\prime }}({{\mathbb{R}}}%
_{v})}^{2},
\end{split}%
\end{equation*}%
where we use the fact that
\begin{equation*}
\left\vert \partial _{\xi }\Big(\langle \xi \rangle ^{2s^{\prime }}\frac{1}{%
1+\delta e^{c_{0}t\langle \xi \rangle ^{2s^{\prime }}}}\Big)\right\vert \leq
C\langle \xi \rangle ^{2s^{\prime }}.
\end{equation*}%
Hence Lemma \ref{lemm3.1} is proved
\end{proof}

Then (\ref{3.1}) and Lemma \ref{lemm3.1} give
\begin{equation}
\frac{1}{2}\frac{d}{dt}\Vert G_{\delta }f(t)\Vert _{L_{1}^{2}({{\mathbb{R}}}%
_{v})}^{2}-\Big(K(f\,,\,\,f),\,\,\,\tilde{f}\Big)_{L^{2}({{\mathbb{R}}}%
_{v})}\leq \,C\Vert G_{\delta }\,f\Vert _{H_{1}^{s^{\prime }}({{\mathbb{R}}}%
_{v})}^{2}.  \label{3.1+11}
\end{equation}

On the other hand, we have
\begin{equation*}
\begin{split}
& \Big(K(f,\,\,f\,),\,\,\tilde{f}\Big)_{L^{2}({{\mathbb{R}}}_{v})}=\Big(%
G_{\delta }K(f,\,\,f\,),\,\,(1+v^{2})\,G_{\delta }f\,\Big)_{L^{2}({{\mathbb{R%
}}}_{v})} \\
=& \Big(K(f,\,\,G_{\delta }f\,),\,G_{\delta }f\,\Big)_{L^{2}({{\mathbb{R}}}%
_{v})}+\Big(K(f,\,\,v\,G_{\delta }f\,),\,v\,G_{\delta }f\,\Big)_{L^{2}({{%
\mathbb{R}}}_{v})} \\
& +\Big(G_{\delta }\,K(f,\,\,f\,)-K(f,\,\,G_{\delta }f\,),\,G_{\delta }f\,%
\Big)_{L^{2}({{\mathbb{R}}}_{v})} \\
& +\Big(v\,G_{\delta }\,K(f,\,\,f\,)-K(f,\,\,v\,G_{\delta
}f\,),\,v\,G_{\delta }f\,\Big)_{L^{2}({{\mathbb{R}}}_{v})}.
\end{split}%
\end{equation*}%
Then Proposition \ref{prop2.1} implies
\begin{equation*}
\left\vert \Big(G_{\delta }\,K(f,\,\,f\,)-K(f,\,\,G_{\delta
}f\,),\,G_{\delta }f\,\Big)_{L^{2}({{\mathbb{R}}}_{v})}\right\vert \leq
C\Vert G_{\delta }f\Vert _{L_{1}^{2}({{\mathbb{R}}}_{v})}\Vert G_{\delta
}f\Vert _{H^{s^{\prime }}({{\mathbb{R}}}_{v})}^{2}
\end{equation*}%
and
\begin{equation*}
\begin{split}
& \left\vert \Big(v\,G_{\delta }\,K(f,\,\,f\,)-K(f,\,\,v\,G_{\delta
}f\,),\,v\,G_{\delta }f\,\Big)_{L^{2}({{\mathbb{R}}}_{v})}\right\vert \\
& \hskip2cm\leq C\,\left( \Vert f\Vert _{L_{1}^{1}({{\mathbb{R}}})}+\Vert {%
G_{\delta }}\,f\Vert _{L_{1}^{2}({{\mathbb{R}}})}\right) \,\Vert {G_{\delta }%
}\,f\Vert _{H_{1}^{s^{\prime }}({{\mathbb{R}}})}^{2}.
\end{split}%
\end{equation*}%
The Proposition \ref{prop2.0} implies
\begin{equation*}
-\Big(K(f,\,G_{\delta }f),G_{\delta }f\Big)\geq c_{f}\Vert G_{\delta }f\Vert
_{H^{s}({{\mathbb{R}}}_{v})}^{2}-C\Vert f\Vert _{L^{1}({{\mathbb{R}}}%
_{v})}\Vert G_{\delta }f\Vert _{L^{2}({{\mathbb{R}}}_{v})}^{2},
\end{equation*}%
\begin{equation*}
-\Big(K(f,\,v\,G_{\delta }f),\,\,v\,G_{\delta }f\Big)\geq c_{f}\Vert
v\,G_{\delta }f\Vert _{H^{s}({{\mathbb{R}}}_{v})}^{2}-C\Vert f\Vert _{L^{1}({%
{\mathbb{R}}}_{v})}\Vert v\,G_{\delta }f\Vert _{L^{2}({{\mathbb{R}}}%
_{v})}^{2}.
\end{equation*}%
Since
\begin{equation*}
\Vert G_{\delta }f\Vert _{H_{1}^{s}({{\mathbb{R}}}_{v})}^{2}\leq \Vert
G_{\delta }f\Vert _{H^{s}({{\mathbb{R}}}_{v})}^{2}+\Vert v\,G_{\delta
}f\Vert _{H^{s}({{\mathbb{R}}}_{v})}^{2}+C\Vert G_{\delta }f\Vert
_{L_{1}^{2}({{\mathbb{R}}}_{v})}^{2}
\end{equation*}%
By summing all the above estimates and (\ref{3.1+11}), we obtain
\begin{equation}
\begin{split}
& \frac{d}{dt}\Vert G_{\delta }f(t)\Vert _{L_{1}^{2}({{\mathbb{R}}}%
_{v})}^{2}+c_{f(t)}\Vert G_{\delta }f(t)\Vert _{H_{1}^{s}({{\mathbb{R}}}%
_{v})}^{2} \\
& \leq C\Vert G_{\delta }f(t)\Vert _{H_{1}^{s^{\prime }}({{\mathbb{R}}}%
_{v})}^{2}+C\Vert f(t)\Vert _{L^{1}({{\mathbb{R}}}_{v})}\Vert G_{\delta
}f(t)\Vert _{L_{1}^{2}({{\mathbb{R}}}_{v})}^{2} \\
& +C\,\left( \Vert f(t)\Vert _{L_{1}^{1}({{\mathbb{R}}})}+\Vert {G_{\delta }}%
\,f(t)\Vert _{L_{1}^{2}({{\mathbb{R}}})}\right) \,\Vert {G_{\delta }}%
\,f(t)\Vert _{H_{1}^{s^{\prime }}({{\mathbb{R}}})}^{2}.
\end{split}
\label{3.1+12}
\end{equation}

\bigbreak\noindent \textbf{End of proof of Theorem \ref{theo3.1}}

By using (\ref{1.3}) and (\ref{1.3b}), we have
\begin{equation*}
\Vert f(t)\Vert _{L^{1}({{\mathbb{R}}}_{v})}+\Vert f(t)\Vert _{L_{2}^{1}({{%
\mathbb{R}}}_{v})}\leq C\Vert f_{0}\Vert _{L_{2}^{1}({{\mathbb{R}}}%
_{v})},\quad c_{f(t)}\geq c_{f_{0}}>0.
\end{equation*}%
Then (\ref{3.1+12}) yields
\begin{equation}
\begin{split}
& \frac{d}{dt}\Vert G_{\delta }f(t)\Vert _{L_{1}^{2}({{\mathbb{R}}}%
_{v})}^{2}+c_{f_{0}}\Vert G_{\delta }f(t)\Vert _{H_{1}^{s}({{\mathbb{R}}}%
_{v})}^{2} \\
& \leq C_{f_{0}}\Vert G_{\delta }f(t)\Vert _{H_{1}^{s^{\prime }}({{\mathbb{R}%
}}_{v})}^{2}+C\,\Vert {G_{\delta }}\,f(t)\Vert _{L_{1}^{2}({{\mathbb{R}}}%
)}\,\Vert {G_{\delta }}\,f(t)\Vert _{H_{1}^{s^{\prime }}({{\mathbb{R}}}%
)}^{2}.
\end{split}
\label{3.1+13}
\end{equation}%
We now need the following interpolation inequality, for $0<s^{\prime }<s$
and any $\lambda >0$,
\begin{equation}
\Vert u\Vert _{H^{s^{\prime }}}^{2}\leq \lambda \Vert u\Vert
_{H^{s}}^{2}+\lambda ^{-\frac{s^{\prime }}{s-s^{\prime }}}\Vert u\Vert
_{L^{2}}^{2}.  \label{interpolation}
\end{equation}%
Then for any small $\varepsilon >0$,
\begin{equation*}
C_{f_{0}}\Vert G_{\delta }f(t)\Vert _{H_{1}^{s^{\prime }}({{\mathbb{R}}}%
_{v})}^{2}\leq \varepsilon \Vert G_{\delta }f(t)\Vert _{H_{1}^{s}({{\mathbb{R%
}}}_{v})}^{2}+C_{\varepsilon ,f_{0}}\Vert G_{\delta }f(t)\Vert _{L_{1}^{2}({{%
\mathbb{R}}}_{v})}^{2}
\end{equation*}%
and
\begin{equation*}
C\,\Vert {G_{\delta }}\,f(t)\Vert _{L_{1}^{2}({{\mathbb{R}}})}\Vert
G_{\delta }f(t)\Vert _{H_{1}^{s^{\prime }}({{\mathbb{R}}}_{v})}^{2}\leq
\varepsilon \Vert G_{\delta }f(t)\Vert _{H_{1}^{s}({{\mathbb{R}}}%
_{v})}^{2}+C_{\varepsilon }\Vert {G_{\delta }}\,f(t)\Vert _{L_{1}^{2}({{%
\mathbb{R}}})}^{\frac{s^{\prime }}{s-s^{\prime }}+2}.
\end{equation*}

We finally get from (\ref{3.1+13}),that for any $0<\varepsilon $ and $%
0<s^{\prime }<s$, there exists $C_{\varepsilon }>0$ such that
\begin{equation*}
\begin{split}
& \frac{d}{dt}\Vert G_{\delta }f(t)\Vert _{L_{1}^{2}({{\mathbb{R}}}%
_{v})}^{2}+(c_{f_{0}}-2\varepsilon )\Vert G_{\delta }f\Vert _{H_{1}^{s}({{%
\mathbb{R}}}_{v})}^{2} \\
& \hskip1cm\leq C_{\varepsilon ,f_{0}}\Vert G_{\delta }f(t)\Vert _{L_{1}^{2}(%
{{\mathbb{R}}}_{v})}^{2}+C_{\varepsilon }\Vert {G_{\delta }}\,f(t)\Vert
_{L_{1}^{2}({{\mathbb{R}}})}^{\frac{s^{\prime }}{s-s^{\prime }}+2}.
\end{split}%
\end{equation*}%
We choose $0<2\varepsilon \leq c_{f_{0}}$, we get
\begin{equation}
\frac{d}{dt}\Vert G_{\delta }f(t)\Vert _{L_{1}^{2}({{\mathbb{R}}}_{v})}\leq
C_{1}\Vert G_{\delta }f(t)\Vert _{L_{1}^{2}({{\mathbb{R}}}_{v})}+C_{2}\Vert {%
G_{\delta }}\,f(t)\Vert _{L_{1}^{2}({{\mathbb{R}}})}^{\frac{s^{\prime }}{%
s-s^{\prime }}+1},\quad t\in \lbrack 0,T_{0}],  \label{3.4}
\end{equation}%
with $C_{1},C_{2}>0$ and independent of $\delta >0$. Then
\begin{equation*}
\frac{d}{dt}\Big(e^{-C_{1}t}\Vert G_{\delta }f(t)\Vert _{L_{1}^{2}({{\mathbb{%
R}}}_{v})}\Big)\leq C_{2}e^{\wt{C}_{1}t}\Big(e^{-C_{1}t}\Vert G_{\delta
}f(t)\Vert _{L_{1}^{2}({{\mathbb{R}}}_{v})}\Big)^{\frac{s^{\prime }}{%
s-s^{\prime }}+1}
\end{equation*}%
where $\wt{C}_{1}=\frac{s^{\prime }\,C_{1}}{s-s^{\prime }}$, thus for $t\in
]0,T_{0}]$
\begin{equation*}
\int_{0}^{t}\frac{d}{d\tau }\Big(e^{-C_{1}\tau }\Vert G_{\delta }f(\tau
)\Vert _{L_{1}^{2}({{\mathbb{R}}}_{v})}\Big)^{-\frac{s^{\prime }}{%
s-s^{\prime }}}d\tau \geq \frac{C_{2}}{C_{1}}\Big(1-e^{\wt{C}_{1}t}\Big).
\end{equation*}%
So that, for $0<\delta <1$,
\begin{equation*}
\Vert G_{\delta }f(t)\Vert _{L_{1}^{2}({{\mathbb{R}}}_{v})}\leq \frac{%
\wt{\wt{C}}_{1}\,e^{C_{1}t}\Vert f_{0}\Vert _{L_{1}^{2}({{\mathbb{R}}}_{v})}%
}{\Big(C_{1}+C_{2}\big(1-e^{\wt{C}_{1}t}\big)\Vert f_{0}\Vert _{L_{1}^{2}({{%
\mathbb{R}}}_{v})}^{\frac{s^{\prime }}{s-s^{\prime }}}\Big)^{\frac{%
s-s^{\prime }}{s^{\prime }}}}\,.
\end{equation*}%

We now choose $0<T_{\ast }\leq T_{0}$ small enough so that
\begin{equation*}
{\Big(C_{1}+C_{2}\big(1-e^{\wt{C}_{1}t}\big)\Vert f_{0}\Vert _{L_{1}^{2}({{%
\mathbb{R}}}_{v})}^{\frac{s^{\prime }}{s-s^{\prime }}}\Big)^{\frac{%
s-s^{\prime }}{s^{\prime }}}}\geq C_{3}>0,\quad t\in \lbrack 0,T_{\ast }],
\end{equation*}%
then by compactness and by taking limit $\delta \,\rightarrow \,0$, we have
for $t\in \lbrack 0,T_{\ast }]$,
\begin{equation}
\Vert e^{c_{0}t\langle D_{v}\rangle ^{2s^{\prime }}}f\Vert _{L^{\infty
}(]0,T_{\ast }[;\,\,L_{1}^{2}({{\mathbb{R}}}_{v}))}^{2}\leq e^{C_{1}T_{\ast
}}\frac{\wt{\wt{C}}_{1}}{C_{3}}\Vert f_{0}\Vert _{L_{1}^{2}({{\mathbb{R}}}%
_{v})}^{2}.  \label{3.5}
\end{equation}%
We therefore have proved Theorem \ref{theo3.1}.

%%%%%%%%%%%%%%%%%%%%%%%%%%%%%%%%%%%%%%%%%%%%%%%%%%%%%%%%%%%%%%%%%%%%%%%
%%%%%%%%%%%%%%%%%%%%%%%%%%%%%%%%%%%%%%%%%%%%%%%%%%%%%%%%%%%%%%%%%%%%%%%
\vskip0.5cm

\section{Radially symmetric Boltzmann equations}

\label{section4}

We consider now the Boltzmann collision operators (\ref{1.5}). In
the Maxwellien case, the Bobylev's formula takes the form
\begin{equation}
{\mathcal{F}}\big(Q(g,\,f)\big)(\xi )=\int_{\mathbb{S}^{2}}b\left( \frac{\xi
}{|\xi |}\,\cdot \,\sigma \right) \left\{ \hat{g}(\xi ^{-})\hat{f}(\xi ^{+})-%
\hat{g}(0)\hat{f}(\xi )\right\} d\sigma \,  \label{4.1}
\end{equation}%
where $\xi\in\mathbb{R}^3$,
\begin{equation*}
\xi ^{+}=\frac{\xi +|\xi |\sigma }{2},\,\,\,\,\,\,\xi ^{-}=\frac{\xi -|\xi
|\sigma }{2}.
\end{equation*}%
On the other hand
\begin{equation*}
|\xi ^{+}|^{2}=|\xi |^{2}\frac{1+\frac{\xi }{|\xi |}\cdot \,\sigma }{2}%
\,,\,\,\,\,\,\,|\xi ^{-}|^{2}=|\xi |^{2}\frac{1-\frac{\xi }{|\xi |}\cdot
\,\sigma }{2}\,,
\end{equation*}%
so that if we define $\theta $ by
\begin{equation*}
\cos \theta =\frac{\xi }{|\xi |}\cdot \,\sigma ,
\end{equation*}%
we obtain
\begin{equation*}
|\xi ^{+}|^{2}=|\xi |^{2}\cos ^{2}\left( \frac{\theta }{2}\right)
\,,\,\,\,\,\,\,|\xi ^{-}|^{2}=|\xi |^{2}\sin ^{2}\left( \frac{\theta }{2}%
\right) \,.
\end{equation*}%
We now consider the radially symmetric function with respect to $v\in {{%
\mathbb{R}}}^{3}$,  namely the function satisfy the property
$$
h(v)=h(Av), \quad v\in{\mathbb{R}}^3
$$
for any  proper orthogonal $3\times 3$ matrix $A$, then $h(v)=h(0, 0, |v|)$.
 Denote by $\mathcal{F}_{\mathbb{R}^3}$ the Fourier
transformation in $\mathbb{R}^3$ and $\mathcal{F}_{\mathbb{R}^1}$
the Fourier transformation in $\mathbb{R}^1$. Then $
\mathcal{F}_{\mathbb{R}^3}(h)(\xi)$  is also radially symmetric with
respect to $\xi \in {{\mathbb{R}}}^{3}$, and it is in the form
$$
\mathcal{F}_{\mathbb{R}^3}(h)(\xi)=\mathcal{F}_{\mathbb{R}^3}(h)(0,
0, |\xi|)=\int_{{\mathbb{R}}} e^{-i |\xi| v_3}
\left(\int_{{\mathbb{R}}^2} h(v_1, v_2, v_3)dv_1 dv_2\right) dv_3.
$$
So that
\begin{equation}\label{4.1+1}
\mathcal{F}^{-1}_{\mathbb{R}^1}
\big(\mathcal{F}_{\mathbb{R}^3}(h)(0, 0, \cdot\,)\big)(u)=
\int_{{\mathbb{R}}^2} h(v_1, v_2, u) dv_1 dv_2\,
\end{equation}
is an even function in $\mathbb{R}$, and we have
\begin{lemma}
\label{lemm2} Assume that $h\in L_{k}^{1}({{\mathbb{R}}}^{3}), h\geq
0$ is a radially symmetric function for certain $k\geq 0$, and
uniformly integrable in  $\mathbb{R}^3$, then
$$
\mathcal{F}^{-1}_{\mathbb{R}^1}
\big(\mathcal{F}_{\mathbb{R}^3}(h)(0, 0, \cdot \,)\big)\in
L_{k}^{1}({{\mathbb{R}}})
$$
is a nonnegative even function, and uniformly integrable in
$\mathbb{R}$.
\end{lemma}

\begin{proof} By using (\ref{4.1+1}), it is evident that $h\in
L^1_k(\mathbb{R}^3)$ implies  $\mathcal{F}^{-1}_{\mathbb{R}^1}
\big(\mathcal{F}_{\mathbb{R}^3}(h)(0, 0, \cdot \,)\big)\in
L_{k}^{1}({{\mathbb{R}}})$, and $h\geq 0$ implies
$\mathcal{F}^{-1}_{\mathbb{R}^1}
\big(\mathcal{F}_{\mathbb{R}^3}(h)(0, 0, \cdot \,)\big)\geq 0$.
Hence we need only to check the uniform integrability of $
\mathcal{F}^{-1}_{\mathbb{R}^1}
\big(\mathcal{F}_{\mathbb{R}^3}(h)(0, 0, \cdot\,)\big)$ in
$\mathbb{R}$. Since $h\in L^{1}({{\mathbb{R}}}^{3})$, for any
$\varepsilon>0$, there exits $R_0>0$ such that
$$
\int_{\{v\in\mathbb{R}^3;\,\, |v|\geq R_0\}} |h(v_1, v_2, v_3)| dv_1
dv_2 dv_3< \frac{\varepsilon}2.
$$
The uniform integrability of $h$ in $\mathbb{R}^3$ imply that, there
exists $\delta_1>$ such that
$$
\int_{B} |h(v_1, v_2, v_3)| dv_1 dv_2 dv_3< \frac{\varepsilon}2,
$$
for any $B\subset \mathbb{R}^3$ with $|B|\leq \delta_1$. Choose new
  $\delta_0=\delta_1 (R_0^2)^{-1}$, then for any $A\subset
\mathbb{R}$, if $|A|\leq \delta_0$, we have
$$
\int_A|\mathcal{F}^{-1}_{\mathbb{R}^1}
\big(\mathcal{F}_{\mathbb{R}^3}(h)(0, 0, \cdot\,)\big)(u)|du \leq
\int_{\mathbb{R}^2\times A} |h(v_1, v_2, v_3)| dv_1 dv_2 dv_3
$$
$$
\leq \int_{\{(v_1, v_2, v_3)\in \mathbb{R}^3;\,\, |v_1|\leq R_0,
|v_2|\leq R_0, v_3\in A\}} |h(v_1, v_2, v_3)| dv_1 dv_2
dv_3+\frac{\varepsilon}2<\varepsilon,
$$
because of
$$
|\{(v_1, v_2, v_3)\in \mathbb{R}^3;\,\, |v_1|\leq R_0, |v_2|\leq
R_0, v_3\in A\}|\leq R_0^2|A|\leq \delta_1.
$$
\end{proof}

\begin{remark}\label{remark5.1}
In the proof of above Lemma, if ~$h\in L\log L(\RR^3)$ then $h$ is uniformly integrable in 
$\RR^3$ with $\delta_1$ depends only on $\varepsilon, \|h\|_{L\log L(\RR^3)}$ and $\|h\|_{L^1(\RR^3)}$.
Therefore, $\mathcal{F}^{-1}_{\mathbb{R}^1}
\big(\mathcal{F}_{\mathbb{R}^3}(h)(0, 0, \cdot\,)$ is uniformly integrable in 
$\RR^1$ with $\delta_0$ also depends only on $\varepsilon, \|h\|_{L\log L(\RR^3)}$ 
and $\|h\|_{L^1(\RR^3)}$.
\end{remark}

\bigbreak\noindent \textbf{End of proof of Theorem \ref{theo2}}

Suppose now $g\in L^{\infty }(]0,+\infty \lbrack ;L_{2+2s}^{1}\cap
L\log L({{\mathbb{R}}}^{3}))$ is a non negative radially symmetric
weak solution of the Cauchy problem (\ref{1.4}). Setting, for $t\geq
0, u\in\mathbb{R}$,
\begin{equation}
f(t, u)=\mathcal{F}^{-1}_{\mathbb{R}^1}
\big(\mathcal{F}_{\mathbb{R}^3}(g)(t, 0, 0, \cdot\,)\big)(u)=
\int_{{\mathbb{R}}^2} g(t, v_1, v_2, u) dv_1 dv_2\, , \label{4.2+1}
\end{equation}
hereafter, the time variable $t$ is always considered as parameters
for the Fourier transformation, then $f(t, u)$ is an even function
with respect to $u\in \mathbb{R}$, and
$$
\hat{f}(t, \tau)=\mathcal{F}_{\mathbb{R}^1} \big(f(t,
\cdot\,)\big)(\tau)=\mathcal{F}_{\mathbb{R}^3}(g)(t, 0, 0, \tau).
$$
So that the Bobylev's formula (\ref{4.1}) give, for $\xi \in
{{\mathbb{R}^3}}$,
\begin{equation}
{\mathcal{F}_{\mathbb{R}^3}}\big(Q(g, g)\big)(\xi )=\int_{-\frac{\pi }{2}}^{\frac{\pi }{2}%
}\beta (|\theta |)\left\{ \hat{f}(t, |\xi| \sin (\theta
/2))\hat{f}(t, |\xi| \cos (\theta /2))-\hat{f}(t, 0)\hat{f}(t, |\xi|
)\right\} d\theta \, \label{4.2}
\end{equation}%
where
\begin{equation*}
\beta (|\theta |)=\frac{1}{2}|\sin \theta |b(\cos \theta ).
\end{equation*}%
Then the right hand side of (\ref{4.2}) is Fourier transformation of
Kac's operator $K(f, f)$. We have proved that if $g(t, v)$ is a non
negative radially symmetric weak solution of the Cauchy problem
(\ref{1.4}), then $f(t, u)$ is a weak solution of the Cauchy problem
of Kac's equation :
\begin{equation}
\left\{
\begin{array}{ll}
\frac{\partial f}{\partial t}(t, u)=K(f, f) (t, u), \\
f(0, u)=f_{0}(u)=\int_{{\mathbb{R}}^2} g_0(v_1, v_2, u) dv_1 dv_2\,,
\end{array}
\right.  \label{5.1}
\end{equation}
or equivalently in the Fourier variable:
\begin{equation*}
\left\{
\begin{array}{ll}
\frac{\partial \hat{f}}{\partial t}(t, \tau)=\int_{-\frac{\pi
}{2}}^{\frac{\pi }{2} }\beta (|\theta |)\left\{ \hat{f}(t, \tau \sin
(\theta /2))\hat{f}(t, \tau \cos
(\theta /2))-\hat{f}(t, 0)\hat{f}(t, \tau )\right\} d\theta , \\
\hat{f}(0, \tau)=\hat{f}_{0}(\tau)=\hat{g}_0(0, 0, \tau).
\end{array}
\right.
\end{equation*}

Under the assumption of Theorem \ref{theo2} for $g(t, v)$, Lemma
\ref{lemm2} and Remark \ref{remark5.1} implies that $f(t, u)$ satisfy the hypothesis of Theorem
\ref{theo1} except $f$ belong to $L\log L$ which substituted by the
uniform integrability of $f=f_{t}(\,\cdot\,)$ in $\mathbb{R}$. As it is point out in
the Remark \ref{rema2.1}, this property is enough to assure the
coercivity (\ref{2.1}). Then we apply Theorem \ref{theo1} to the
Cauchy problem (\ref{5.1}), thus there exists $T_*>0$ such that for
$0<t\leq T_*$,
\begin{equation*}
e^{c_{0}t\langle \,|\tau |\,\rangle ^{2s^{\prime }}}\hat{f}(t, \tau
)=e^{c_{0}t\langle \,|\tau |\,\rangle ^{2s^{\prime
}}}\mathcal{F}_{\mathbb{R}^3}(g)(t, 0, 0, \tau )\in
H^{1}({{\mathbb{R}}}_{\tau }).
\end{equation*}
It remain to prove the Gevrey smoothing effect in the global time
interval. Kac's equation shares with the homogeneous Boltzmann
equation for Maxwellian molecules the existence and uniqueness
theory for the Cauchy problem, see \cite{tos-villani} for the
uniqueness of weak solution for the non-cut-off Boltzmann equation.
We take $ 0<t_{0}<t_{1}\leq T_{\ast }$, and consider the Cauchy
problem (\ref{5.1}) with even initial datum $\hat{f}(t_1,\tau)$. The
Sobolev embedding $H^{1}({{\mathbb{R}}}) \subset L^{\infty
}({{\mathbb{R}}})$ imply that
\begin{eqnarray*}
\Vert e^{c_{0}t_{1}\langle \,\cdot \,\rangle ^{2s^{\prime }}}\hat{f}%
(t_{1},\,\cdot \,)\Vert _{L^{\infty }({{\mathbb{R}}})}&\leq& C\Vert
e^{c_{0}t_{1}\langle \,\cdot \,\rangle ^{2s^{\prime
}}}\hat{f}(t_{1},\,\cdot \,)\Vert _{H^{1}({{\mathbb{R}}})}\\
& \leq& C\Vert e^{c_{0}t_{1}\langle \,|D_u| \,\rangle ^{2s^{\prime
}}}f(t_{1},\,\cdot \,)\Vert _{L^{2}_1({{\mathbb{R}}})} <+\infty
\end{eqnarray*}
Now the following propagation of Gevrey regularity results deduces
the Gevrey smoothing effect in the global time interval.

\begin{theorem}
\label{theo3.2} \textbf{(Theorem 2.3 of \cite{Des-Fu-Ter})}

Let $f_{0}$ be a non negative, even function, satisfying
\begin{equation*}
\sup_{\xi \in {{\mathbb{R}}}}\Big(|\hat{f}_{0}(\xi )|e^{c_{1}\langle
\xi \rangle ^{2s}}\Big)<+\infty ,
\end{equation*}%
for some $c_{1}>0$ and  the cross-section $\beta $ satisfying
(\ref{1.1+1}) with $0<s<1$. Then the solution of the Cauchy
problem (\ref{1.1}) satisfies $f(t,\,\cdot \,)\in G^{^{\frac{1}{2s}}}({{\mathbb{R}}}%
))$ for any $t\geq 0$.
\end{theorem}

In conclusion, if $g\in L^{\infty }(]0,+\infty \lbrack
;L_{2+2s}^{1}\cap L\log L({{\mathbb{R}}}^{3}))$ is a non negative
radially symmetric weak solution of the Cauchy problem (\ref{1.4}),
then under the assumption of Theorem \ref{theo2}, we have proved
that for any fixed $0<t<+\infty$, there exists $c_0>0$ such that
$$
e^{c_{0}\langle |D_u|\rangle ^{2s^{\prime }}} f(t, \,\cdot\,)\in
L^2(\mathbb{R})
$$
where $f$ is the function defined by (\ref{4.2+1}). We can finish
now the proof by the following estimations, for fixed $t>0$,
\begin{eqnarray*}
\|e^{\frac{c_{0}}{2}\langle |D_v|\rangle ^{2s^{\prime }}} g(t,
\,\cdot\,)\|^2_{L^2(\mathbb{R}^3)}&=&
\int_{\mathbb{R}^3}\left|e^{\frac{c_{0}}{2}\langle |\xi|\rangle
^{2s^{\prime }}} \mathcal{F}_{\mathbb{R}^3}(g)(t, \xi_1, \xi_2,
\xi_3)\right|^2d\xi
\\
&=&\int_{\mathbb{R}^3}\left|e^{\frac{c_{0}}{2}\langle |\xi|\rangle
^{2s^{\prime }}}\mathcal{F}_{\mathbb{R}^3}(g)(t, 0, 0,
|\xi|)\right|^2d\xi\\
&=&C\int^\infty_{0}\left|e^{\frac{c_{0}}{2}\langle \tau\rangle
^{2s^{\prime }}} \hat{f}(t, \tau)\right|^2 \tau^2 d\tau
\\
&\leq& C\int^\infty_{0}\left|e^{c_{0}\langle \tau\rangle
^{2s^{\prime }}} \hat{f}(t, \tau)\right|^2 d\tau\\
&\leq& C \|e^{c_{0}\langle |D_u|\rangle ^{2s^{\prime }}} f(t,
\,\cdot\,)\|^2_{L^2(\mathbb{R})}<+\infty.
\end{eqnarray*}
We finished the proof of Theorem \ref{theo2}.

\end{document}